\documentclass[12pt]{amsart}
\usepackage[latin1]{inputenc}
\usepackage{graphics}
\usepackage{graphicx}
\usepackage{psfrag}
\usepackage{amsmath}
\usepackage{amssymb}
\usepackage[all]{xy}

\theoremstyle{plain}
\newtheorem{thm}{Theorem}[section]
\numberwithin{equation}{section} 
\numberwithin{figure}{section}
\theoremstyle{plain}
\newtheorem{corollary}[thm]{Corollary}
\theoremstyle{plain}
\newtheorem{lemma}[thm]{Lemma}
\theoremstyle{plain}
\newtheorem{prop}[thm]{Proposition}
\theoremstyle{definition}
\newtheorem{defn}[thm]{Definition}
\theoremstyle{remark}
\newtheorem{remark}[thm]{Remark}
\theoremstyle{remark}
\newtheorem*{remark*}{Remark}
\theoremstyle{remark}

\theoremstyle{plain}

\theoremstyle{plain}

\theoremstyle{plain}

\theoremstyle{plain}

\newcommand{\noun}[1]{\textsc{#1}}
\newcommand{\len}{\operatorname{len}}

\begin{document}
\title[Counting closed geodesics]{Counting closed geodesics
on rank one manifolds}
\author{Roland Gunesch}
\address{Universit\"at Hamburg, Department Mathematik, 
SP Dgl. \& Dynami\-sche Systeme, 
Bundesstr. 55, 20146 Hamburg, Germany}
\email{gunesch{\relax}@{\relax}math{.}uni-hamburg{.}de}
\date{}
\begin{abstract}
We establish a precise asymptotic formula for the number of homotopy classes of
periodic
orbits for the geodesic flow on rank one manifolds of nonpositive curvature.
This extends a celebrated result of  G.~A.~Margulis to the nonuniformly
hyperbolic
case and strengthens previous results by G.~Knieper.

We also establish some useful properties of the measure of
maximal entropy.
\end{abstract}


\thanks{The author thanks Anatole B. Katok for
suggesting
this problem and for valuable comments,  helpful discussions and precious
support.
Gerhard Knieper provided helpful discussions during the
creation of this work.
Parts of this work were completed while the author was at the Isaac Newton
Institute
for the Mathematical Sciences, Cambridge, and at the ETH Z\"urich.}

\maketitle

\section{Introduction}

\subsection{Manifolds of rank one}

Let \( M \) be a compact Riemannian manifold with all sectional curvatures
nonpositive. 
For a
vector \( v\in TM, \)
the {\bf rank} of \( v \) is the dimension of the vector space of parallel 
Jacobi fields along the geodesic tangent to \( v. \) The rank of \( M \)
is the minimal rank of all tangent vectors. 
Obvious consequences of this definition are that 
\[
1\leq \textrm{rank}(M)\leq \dim (M),\]
 that the rank of \( \mathbb {R}^{k} \) with the flat metric is \( k \) and
that 
\[
\textrm{rank}(M\times N)=\textrm{rank}(M)+\textrm{rank}(N).\]
Every manifold whose sectional curvature is never zero is automatically of rank
one. Products with Euclidean \( n \)-space clearly have rank at least \( n+1.
\)
However, it is possible for a manifold to be everywhere locally a product with
Euclidean space and still have rank one. 
It turns out that the rank of a manifold of nonpositive curvature is the
algebraic
rank of its fundamental group (\cite{BaEb}). 

Apart from manifolds of negative curvature, examples are nonpositively curved
surfaces containing flat cylinders or an infinitesimal analogue of a flat
cylinder, as illustrated in the following diagram.
\begin{figure}[h]
\begin{center}
\footnotesize
\includegraphics[width=0.5\textwidth]{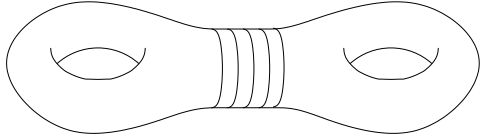}
\label{rksurf}\caption{A surface of rank one with a flat strip and a
parallel family
of geodesics.}
\end{center}
\end{figure}
\normalsize

 In higher dimensions,
examples
include M. Gromov's (3-dimensional) graph manifolds (\cite{gro}). There is an
interesting
rigidity phenomenon: Every compact 3-manifold of nonpositive curvature whose
fundamental group is isomorphic to that of a graph manifold is actually
diffeomorphic
to that graph manifold (\cite{Sch}).

We will study properties of manifolds of rank one in this article.

\subsection{Reasons to study these spaces }
\subsubsection{Rank rigidity}

W.~Ballmann (\cite{higher rk}) and independently K.~Burns and R.~Spatzier
(\cite{BuSp})
showed that the universal cover of a nonpositively curved manifold can be
written
uniquely as a product of Euclidean, symmetric and rank one spaces. The first
two types are understood, due to P.~Eberlein and others. (A general introduction
to higher rank symmetric spaces is e.g. \cite{Ebe5}; see also \cite{5}. For
a complete treatment of rank rigidity, see \cite{dmv}.) 

Thus, in order to understand nonpositively curved manifolds, the most relevant
objects to examine are manifolds of dimension at least two with rank
one. This becomes
even more obvious if one considers the fact that rank one is generic
in nonpositive
curvature (\cite{bbe}). Thus, in a certain sense, ``almost all'' nonpositively
curved manifolds have rank one.

\subsubsection{Limits of hyperbolic systems}

Another reason to study nonpositively curved manifolds is the following. On
one hand, strongly hyperbolic systems, particularly geodesic flows on compact
manifolds of negative curvature, are well understood since D.~V.~Anosov
(\cite{Ano}, \cite{Mar2}, \cite{Mar3}). 
Later, P.~Eberlein established a condition weaker than
negative
curvature which still ensures the Anosov property of the geodesic flow
(\cite{Ebe3 anosov}, \cite{Ebe4 anosov}).
Also, Ya.~Pesin and M.~Brin extended the notion of hyperbolicity to that of
partial hyperbolicity (\cite{Pes}, \cite{BrPe1}, \cite{BrPe2}).
On the other hand, much less is known about the dynamics of systems lacking
strong hyperbolicity. 
The open
set of geodesic flows on manifolds with negative curvature is ``essentially''
understood (hyperbolicity is an open property), and hence the edge
of our knowledge 
about such flows
is mainly marked by the boundary of this set, which is a set
of geodesic
flows on manifolds of nonpositive curvature. Therefore it is important to
study the dynamics of these.

However, the set of nonpositively curved manifolds is larger than just the
closure
of the set of negatively curved manifolds. This can be seen e.g.  as follows:
Some nonpositively
curved manifolds, such as Gromov's graph manifolds, contain an embedded
2-torus.
Thus their fundamental group contains a copy of \( \mathbb {Z}^{2} \). Hence,
by Preissmann's theorem, they do not admit \emph{any} metric of negative
curvature.
Therefore, the investigation in this article actually deals with even
more than
the limits of our current knowledge of strongly hyperbolic systems.

\subsection{Statement of the result }
We count homotopy classes of closed geodesics ordered by length in the
following sense:
The number $P_t$ of homotopy  classes of periodic orbits
of length at most \( t \) is finite for all \( t .\) (For a
periodic geodesic there may be
uncountably many periodic geodesics homotopic 
to it, but  in nonpositive curvature they all have the same length.)

Trying to find a concrete and explicit formula for $P_t$ which is accurate for
all values of $t$
is completely hopeless, even on very simple manifolds. Nonetheless, in this
article we manage to derive an asymptotic formula for $P_t$, i.e. a formula
which tells us the behavior of
$P_t$  when $t$ is large.
We will show (Theorem \ref{thm:P asym e^ht/ht}):
\[
P_{t}\sim \frac{1}{ht}e^{ht}\]
 where the notation \( f(t)\sim g(t) \) means \( \frac{f(t)}{g(t)}\rightarrow 1
\)
as \( t\rightarrow \infty . \)
This extends a celebrated result of G.~A.~Margulis to the case
of nonpositive curvature. It also strengthens results by G.~Knieper, which were
the
sharpest estimates known to this date in the setup of nonpositive curvature.
This is 
explained in more detail in the following section. 

\section{History}

\subsection{Margulis' asymptotics}

The study of the functions \( P_. \) and \( b_. \), where \( b_{t}(x) \) is the
volume of the geodesic ball of radius \( t \) centered at \( x \),  was
originated by G.~A.~Margulis in his dissertation 
\cite{Mar1}. He covers the case
where the curvature is strictly negative.
His influential results were published in \cite{Mar2} and the proofs
were published eventually in
\cite{Mar4}.
He established that, on a compact manifold of negative curvature, 
\begin{equation}
\label{marg asym for vol}
b_{t}(x)\sim c(x)e^{ht}
\end{equation}
 for some continuous function \( c \) on \( M. \) He also showed that 
\begin{equation}
\label{marg asym for P}
P_{t}\sim c'\frac{e^{ht}}{t}
\end{equation}
for some constant $c'$.
In modern notation, the exponent \( h \) is the topological entropy of the geodesic flow. 
See \cite{kh} for a modern reference on the topic of entropy. 

Margulis  pointed out that if the curvature is constant with value $K$
then the exponential growth rate equals \( (n-1)\sqrt{-K} \) and that in this
case the function \( c \) is constant. In fact, \( c\equiv
{1}/{h} \). Moreover, \( c'=1/h \) for variable negative curvature.

\subsection{Beyond negative curvature; Katok's entropy conjecture }

The vast majority of the studies that have since been done are restricted to
negative
curvature; see e.g. \cite{PaPo}, \cite{BaLe}, \cite{PoSh1}.
 The reason is that in that case techniques from uniformly hyperbolic
dynamics can be applied. From the point of view of
analysis, this case is
much easier to treat. 
However, from a geometrical viewpoint, manifolds of nonpositive
curvature are a natural object to study. Already in the seventies the
investigation
of manifolds of nonpositive curvature became the focus of interest of
geometers. (Also more
general classes have been studied since, such as manifolds without \emph{focal
points}, i.e.~where every parallel Jacobi field with one zero has the property
that its length increases monotonically when going away from the zero, or
manifolds
without \emph{conjugate points}, i.e. such that any Jacobi field with two
zeroes
is trivial.) In 1984 at a MSRI problem session 
a major list of problems which were open at the time was
compiled (\cite{BuKa}), including A.~Katok's \emph{entropy} \emph{conjecture:}
The measure
of maximal entropy is unique.

One of the first result in the direction of asymptotics of closed geodesics in
nonpositive curvature is 
the fact that
the growth rate of closed geodesics equals the topological entropy 
\( h \) , even if the curvature is just
nonpositive (instead of strictly negative). 
G.~Knieper calculated the growth rate of closed geodesics
in \cite{knie arch d math}.
This result can also be deduced 
from A.~Manning's result \cite{mann ent}
that the growth rate of volume equals \( h \) in nonpositive curvature.

This shows in particular that the exponent in Margulis'
asymptotics must equal \( h \) (we have already written
Margulis' equations that way). A method for showing that in the case of 
strictly negative curvature the constant \( c' \) in
equation (\ref{marg asym for P}) equals \( 1/h \) is outlined in C. Toll's
dissertation \cite{Tol} and published in \cite{kh}. This method was developed
by Margulis in his thesis \cite{Mar1} and published in \cite{Mar4}.
The behavior of the function \( c \) in the asymptotic
formula (\ref{marg asym for vol}) was investigated by C.~Yue in \cite{Y jdg}
and \cite{Y etds}. For recent developments concerning the 
asymptotics of the number of homology classes see e.g.
the works of N.~Anantharaman (\cite{Ana1}, \cite{Ana2}), M.~Babillot and
F.~Ledrappier (\cite{BaLe}), M.~Pollicott and R.~Sharp (\cite{PoSh2}), 
and S.~P.~Lalley (\cite{Lal}).

It took almost two decades after Knieper's and Manning's results, which in
turn were published about one decade after Margulis' results, until the next
step in the analysis of asymptotics of periodic orbits on manifolds of
nonpositive curvature was completed, again by
Knieper.

\subsection{Knieper's multiplicative bounds}\label{kniepersection}

In 1996 G.~Knieper proved  asymptotic multiplicative
bounds for 
volume
and periodic orbits 
(\cite{knie gafa}) which, in the case of nonpositive curvature
and rank one,  were the sharpest results known
until now: There exists a constant \( C \) such that for sufficiently large
\( t \), 
\[
\frac{1}{C}\leq \frac{s_{t}(x)}{e^{ht}}\leq C,\]
where $s_t(x)$ is the volume of the sphere of radius $t$ centered at $x$, and 
\[
\frac{1}{Ct}\leq \frac{P_{t}}{e^{ht}}\leq C.\]
The main step
in the proof of these asymptotics is the proof of Katok's entropy conjecture.
Knieper also demonstrated in \cite{knie ann} that the measure of maximal
entropy can be obtained
via the Patterson-Sullivan construction (\cite{Pat}, \cite{Sul}; see also 
\cite{Kai1}, \cite{Kai2}).
Moreover, for the case of higher rank  Knieper obtained asymptotic
information
using rigidity. Namely, \[
\frac{1}{C}\leq \frac{s_{t}}{t^{(\text{rank}(M)-1)/2}e^{ht}}\leq C.\]
He also estimates the number of closed geodesics in higher rank.

Knieper subsequently sharpened his results. With the same method he is  able 
to prove
that in the rank one case actually $$\frac{1}{C}\leq
\frac{P_{t}}{e^{ht}/t}\leq {C}$$ holds (see \cite{kni4}).
Still, the quotient of the upper and lower 
bounds is a constant which cannot be made close to 1.

The question whether in this setup of nonpositive curvature and rank
one
one can prove more
precise multiplicative asymptotics---namely without such multiplicative
constants---has remained open so far. In this article we establish this result.

\begin{remark}
For non-geodesic dynamical systems no statements providing asymptotics similar
the
ones mentioned here are known. 
 One
of the best known results is that for some prevalent set of 
diffeomorphisms the number of periodic orbits of period \( n \) is bounded
by \( \exp(C\cdot n^{1+\delta }) \) for some \( \delta >0 \) (\cite{KaHu}).

But even for geodesic flows in the absence of nonpositive
curvature it is difficult to count---or even find---closed geodesics.
The fact that every compact manifold has even \emph{one} closed geodesic
was established only in 1951 by Lyusternik and Fet (\cite{LuFe}).
In the setup of positively curved manifolds and their kin, one of the strongest
known results is H.-B.~Rademacher's Theorem from 1990 (\cite{Rad1}, 
\cite{Rad2})
stating that every 
connected and simply connected compact manifold has infinitely many
(geometrically distinct) closed
geodesics for a \( C^{r} \)-generic metric for all \(  r\in  
[2,\infty ]. \) See also \cite{Rad3} for this.

For Riemannian metrics on the 2-sphere, existence of many closed geod\-esics took
considerable effort to prove.  The famous Lyusternik-Shnirelman Theorem
asserts the existence of three (geometrically distinct) closed geod\-esics. The
original proof in
\cite{LuSch} is considered to have gaps. 
Complete proofs were given by W.~Ballmann (\cite{Bal3}), 
W.~Klingenberg (with W.~Ballmann's help) (\cite{Kli})
and also J.~Jost (\cite{JosSphere}, \cite{JosCorr}). 
See also \cite{BTZ1},  \cite{BTZ2}.

J.~Franks (\cite{Fra}) established that every metric of positive curvature on
$S^2$ has
infinitely many (geometrically distinct) geodesics. This is a consequence of
his results about area-preserving annulus homeomorphisms.  
V.~Bangert managed to show existence of infinitely many 
(geometrically distinct) geodesics on $S^2$
without requiring positive curvature by means of variational methods 
 (\cite{Ban}).

For the case of Finsler manifolds, there  actually exist examples of simply
connected manifolds that possess only finitely many geometrically distinct
closed geodesics. On
$S^2$ such examples were constructed by A.~B.~Katok in \cite{Kat1} as a
by-product of a more general construction.
Explaining this particular aspect of Katok's construction is also
the topic of \cite{Mat}. \cite{Zil} also studies the Katok examples.

\end{remark}

In this article we derive asymptotics like the ones Margulis obtained. 
We
prove them for nonpositive curvature and rank one 
using non-uniform hyperbolicity.
Hence the same strong statement is true in considerably greater generality.

\section{Geometry and dynamics in nonpositive curvature}
Let $M$ be a compact rank one Riemannian manifold of
nonpositive curvature.   As is usual, we
assume it to be  connected and geodesically complete.
Let \( S\tilde{M} \) be the unit sphere bundle of the universal covering
of $M$. 
For \( v\in  S\tilde{M}  \) let \( c_{v} \) be the geodesic satisfying \( c'(0)=v \)
(which is hence automatically parameterized by arclength). Here \( c' \) of
course denotes the covariant derivative  of \( c. \) 
Let \( g=(g^{t})_{t\in \mathbb {R}} \) 
be the \textbf{geodesic flow} on \(  S\tilde{M} ,
\)
which is defined by \( g^{t}(v):=c_{v}'(t)=:v_t. \) 
\subsection{Review of asymptotic geometry}
\begin{defn} Let \( \pi : \) \( TM\rightarrow M \) be the canonical projection. 
We say that \( v,w\in  S\tilde{M}  \) are \textbf{positively asymptotic} (written 
 \( v\sim w \)) if there exists a constant \( C \) such that \( 
d(\pi g^{t}v,\pi g^{t}w)<C \)
for all \( t>0 \). 
This is evidently an equivalence relation. 
Similarly,  \( v,w\in  S\tilde{M}  \) are \textbf{negatively asymptotic}
if
 $-v\sim-w$.
\end{defn}
Recall that \( \textrm{rank}(v):=\dim\{\textrm{parallel Jacobi fields along
}c_v\}. \)
Clearly the rank is constant along geodesics, i.e. 
\(
\textrm{rank}(c_{v}'(t))=\textrm{rank}(c_{v}'(0))\)
 for all \( t\in \mathbb{R} . \) 
\begin{defn}
We call a vector \( v\in  S\tilde{M}  \), as well as the geodesic \( c_{v}, \)
\textbf{regular}
if \( \textrm{rank}(v)=1 \) and \textbf{singular} if \( \textrm{rank}(v)>1. \)
\def\Reg{{\bf{Reg}}}
\def\Sing{{\bf{Sing}}}
Let \( \textrm{\Reg} \) and \( \textrm{\Sing} \) be the sets of regular and
singular
vectors, respectively. 
\end{defn}
\begin{remark}
The set \( \textrm{{\bf{Reg}}} \) is open
since rank is semicontinuous in the sense that 
\(
\textrm{rank}(\lim _{n}v_{n})\geq \lim _{n}\textrm{rank}(v_{n}).\)
\end{remark}
\begin{remark}
For every \( v\in  S\tilde{M}  \) and every \( p\in \tilde{M} \) 
there exists some \( w_{+}\in
S_{p}\tilde{M} \)
which is positively asymptotic to \( v \) and some \( w_{-}\in S_{p}\tilde{M}
\)
which is negatively asymptotic to \( v. \) In contrast, the existence of \(
w_{+-}\in T_{p}\tilde{M} \)
which is simultaneously positively \emph{and} negatively asymptotic to \( v \)
is rare. Moreover, if \( v\sim w \) and \(
-v\sim -w \) 
then \( v,w \) bound a flat strip, i.e. a totally geodesic embedded copy of
\( [-a,a]\times \mathbb {R} \) with Euclidean metric. Here the number \( a \)
 is positive if \( v,w \) do not lie on the same geodesic
 trajectory. In particular,
if \( \textrm{rank}(v)=1 \) (hence $c_v$ is a regular geodesic) then there
does not exist such $w$ with \( w\sim v \)
and $-v\sim -w$
through \emph{any} base point in the manifold outside \( c_{v}. \) In other
words, if \( w\sim v \) and $-w\sim -v$
 on a rank 1 manifold then \( w=g^{t}v \) for
some \( t. \) 
On the other hand, if  \( \textrm{rank}(v)>1 \) (and thus $c_v$ is a singular
geodesic) then $v$ and hence $c_v$ may lie in a flat strip of positive width,
and in that
case there are vectors $w$ with  \( w\sim
v \) and $-w\sim -v$ at base points outside $c_v,$
namely at all base points in that flat strip.
\end{remark}
Since \( \tilde{M} \) is of nonpositive curvature, it is diffeomorphic to \( \mathbb{R}
^{n} \)
by the Hadamard-Cartan theorem, hence to an open Euclidean $n$-ball. It admits
the
\textbf{compactification} 
\(
\overline{M}=\tilde{M}\cup \tilde{M}(\infty )\)
 where \( \tilde{M}(\infty ), \) the \textbf{boundary at infinity} of \(
\tilde{M}, \)
is the set of equivalence classes of positively asymptotic vectors, i.e., 
\(
\tilde{M}(\infty )= S\tilde{M} /\sim .\)

A detailed description of spaces of
nonpositive curvature, even without a manifold structure,
can be found in \cite{dmv}. 

\subsection{Stable and unstable spaces}
\begin{defn}
Let \( \mathcal{K}:TS\tilde{M}\to S\tilde{M} \)
be the \textbf{connection map}, i.e. 
\(
\mathcal{K}\xi :=\nabla _{d\pi \xi }Z\)
 where $\nabla$ is the Riemannian connection and 
\(
Z(0)=d\pi \xi \), \( \left.\frac{d}{dt}Z(t)\right|_{t=0}=\xi .\)
 We obtain a Riemannian metric on \( SM \), the \textbf{Sasaki
metric},
by setting \( \langle\xi ,\eta \rangle :=\langle d\pi \xi ,d\pi \eta \rangle
+\langle \mathcal{K}\xi ,\mathcal{K}\eta \rangle  \)
for $\xi,\eta\in T_vSM$ where $v\in SM$. Hence we can talk about length of
vectors in $TS\tilde{M}$.
\end{defn}
There is a canonical isomorphism 
\(
(d\pi ,\mathcal{K})\)
 between \( T_{v}SM \) and the set
of Jacobi fields
along \( c_{v}. \) It is given by \( \xi \mapsto J_\xi \) with 
$J_\xi(0)=d\pi \cdot \xi ,\ J'_\xi(0)=\mathcal{K}\xi .$
This uses the well-known fact that a Jacobi field is determined by
its value and derivative at one point.

The space \( TS\tilde{M} \), i.e. the tangent bundle of the unit
sphere
bundle, admits a natural splitting \[ TS\tilde{M}=E^{s}\oplus E^{u}\oplus
E^{0}, \] i.e. \( T_{v}S\tilde{M}=E_{v}^{s}\oplus E_{v}^{u}\oplus E_{v}^{0} \)
for all
$v\in S\tilde{M}$, where
\begin{eqnarray*}
E_{v}^{0}&:=&\mathbb {R}\cdot
\left.\frac{d}{dt}g^{t}v\right|_{t=0},\\
E^{s}_{v}&:=&\{\xi \in T_{v}S\tilde{M}:\xi \perp E^{0},\, J_{\xi
}\, 
\textrm{is the stable Jacobi field along }d\pi \xi \},\\
E^{u}_{v}&:=&\{\xi \in T_{v}S\tilde{M}:\xi \perp E^{0},\, J_{\xi
}\, 
\textrm{is the unstable Jac. field along }d\pi \xi \}.
\end{eqnarray*}
\begin{defn}
For \( v\in S\tilde{M}, \) define \( W^{s}(v), \) the \textbf{stable
horosphere} based
at \textbf{\( v, \)} to be the integral manifold of the distribution \( E^{s}
\)
passing through \( v. \) Similarly, define \( W^{u}(v), \) the \textbf{unstable
horosphere} based at \textbf{\( v \),} via integrating \( E^{u}. \) The
projection of $W^s$ (resp. $W^u$) to $\tilde{M}$ is again 
called the stable horosphere (resp. the unstable horosphere).
The flow
direction of course integrates to a
geodesic
trajectory, which one might call \(
W^{0}(v). \) The $0$- and $u$-directions are jointly integrable, giving rise
to an integral manifold $W^{0u}$, and similarly the $0$- and $s$-directions
give rise to an integral manifold $W^{0s}$. 
We write \( B^{i}_{\delta } \) (resp. \(\overline{ B^{i}_{\delta }} \)) for the open
(resp. closed) \( \delta  \)-neighborhood in \( W^{i} \)
(\( i=u,s,0u,0s,0 \)). 
\end{defn}
On
the other hand, the
$u$- and $s$-directions are usually not jointly integrable.
Continuity of these foliations has been proven in this form by P. Eberlein
(\cite{Ebe2 visibility})
and J.-H. Eschenburg (\cite{Esch}): 
\begin{thm}\label{thm:unif cts}
Let \( M \) be a compact manifold of nonpositive curvature. Then the foliation
\( \{W_{}^{s}(v):v\in S\tilde{M}\} \) of \( S\tilde{M} \) by stable horospheres
is 
continuous.
 The same holds for the foliation \( \{W_{}^{u}(v):v\in S\tilde{M}\} \)
of \( S\tilde{M} \) by unstable horospheres.
\end{thm}
Note that due to compactness of $M$ (hence of $SM$), the continuity is
automatically uniform.

During the same years, Eberlein
considered similar questions on \emph{Visibility manifolds}  
(\cite{Ebe2 visibility}).
 The continuity 
result was improved by M. Brin (\cite[Appendix A]{BaPe}) to H\"olderness on the
Pesin sets; see \cite{BaPe} for the definition of these sets.
For our discussion, uniform continuity is sufficient.

The following result is easier to show in the hyperbolic case (i.e. strictly
negative
curvature) than for nonpositive curvature, where it  is a major
theorem,  proven
by Eberlein (\cite{Ebe isom}):

\begin{thm}
Let \( M \) be a compact rank one manifold of nonpositive curvature. Then
stable
manifolds are dense. Similarly, unstable manifolds are dense.
\end{thm}

\subsection{Important measures}

The Riemannian structure gives rise to a natural measure \( \lambda  \) on
\( SM, \) called the \textbf{Liouville measure}. It is finite since \( M \) is
compact. It is the prototypical smooth measure, i.e., for any smooth chart 
\( \varphi :\: U\rightarrow \mathbb{R} ^{2n-1}, \) $U\subset SM$ open,
the measure \( \varphi _{*}\lambda  \) on a subset of \( \mathbb{R} ^{2n-1} \) is
smoothly
equivalent to Lebesgue measure. 

The well-known variational principle (see e.g. \cite{kh}) asserts that the
supremum of the entropies of invariant probability measures on \( SM \) is
the topological entropy \( h \). The variational principle by itself of course
guarantees neither existence nor uniqueness of a \textbf{measure of maximal
entropy}, i.e. one whose entropy actually
equals \( h. \) These two facts were established in the setup of nonpositive
curvature by Knieper (\cite{knie ann}): 
\begin{thm}
There is a measure of maximal entropy for the geodes\-ic flow on a compact rank
one nonpositively
curved manifold. Moreover, it is unique.
\end{thm}
The proof uses the Patterson-Sullivan
construction (\cite{Pat}, \cite{Sul}; see also \cite{Kai1}, \cite{Kai2}). 
Knieper's construction builds the 
measure as limit of measures supported on periodic orbits. 

For the case of strictly negative curvature, the measure of maximal
entropy was previously 
constructed (in a different way) by Margulis (\cite{Mar3}).
He used it to obtain his asymptotic results.
His construction builds the measure as the product of limits of measures
supported on pieces
of stable and unstable leaves. The measure thus obtained is hence called the 
\textbf{Margulis measure}. It agrees with the \textbf{Bowen measure} which is
obtained as limit of measures concentrated on periodic orbits. U.~Hamenst{\"a}dt
(\cite{Ham}) gave a geometric description of the Margulis measure by projecting
distances
on horospheres to the boundary at infinity, and this description was
immediately
generalized to Anosov flows by B.~Hasselblatt (\cite{Has}). 

The measure of
maximal entropy is adapted to the dynamical properties of the flow. In
particular, we will see that the conditionals of this
 measure show uniform expansion/contraction with time. In negative
curvature, this can be seen by considering the Margulis measure, where this
 property is a natural by-product of the construction. In nonpositive
 curvature, however, this property is not  immediate. 
 We show it in Theorem \ref{thm:unif-exp-of-cond}.
 
The measure of maximal entropy is sometimes simply called \textbf{maximal
measure}.
In the setup of nonpositive curvature, the name \textbf{Knieper
measure} could be appropriate.

\begin{remark}
It is part of Katok's entropy conjecture and shown in \cite{knie ann} that
\( m(\textrm{{\bf{Sing}}})=0 \) (and in fact even that \(
h(g|_{\textrm{{\bf{Sing}}}})<h(g) \)).
In contrast, whether \( \lambda (\textrm{{\bf{Sing}}})=0 \) or not is a major
open question;
it is equivalent to the famous problem of ergodicity of the geodesic flow
in nonpositive curvature with respect to the Liouville measure $\lambda$. 
On the other hand,
ergodicity of the geodesic flow in nonpositive curvature with respect to \( m\) 
has been proven by Knieper. 
\end{remark}
A very useful dynamical property is mixing, which implies ergodicity. For
nonpositive
curvature mixing has
been proven by M. Babillot (\cite{Bab}): 
\begin{thm}
The measure of maximal entropy for the geodesic flow on a compact rank one
nonpositively
curved manifold is mixing.
\end{thm}
We use this property in our proof of the asymptotic formula.
\subsection{\label{par jac flds}Parallel Jacobi fields}
\begin{lemma}
\label{lem:reg iff transversal}The vector \( v\in SM \) is regular if and only
if \( W^{u}({v}) \), \( W^{s}({v}) \) and $W^0(v)$ intersect 
transversally at \(
v. \)
\end{lemma}
Here transversality of the three manifolds
means that 
\[
T_{v}SM=T_{v}W^{u}\oplus T_{v}W^{0}\oplus T_{v}W^{s}.\]
\begin{proof}
\( W^{u}({v}) \) and \( W^{s}({v}) \) intersect with zero angle
at \( v \) if and only if there exist 
\[
\xi \in TW_{}^{u}(v)\cap TW_{}^{s}(v)\subset T_{v}SM.\]
 But \( \xi \in TW_{}^{s}(v) \) is true if and only if \( J_{\xi } \) is the
stable
Jacobi field along \( c_{v} \), and \( \xi \in TW_{}^{u}(v) \) is true if and
only
if \( J_{\xi } \) is the unstable Jacobi field along \( c_{v} \). A Jacobi
field \( J \) is both the stable and the unstable Jacobi field along \( c_{v}
\)
if and only if \( J \) is parallel. The nonexistence of such \( J \)
perpendicular to \( c_{v} \) is just
the definition of rank one.
\end{proof}

\subsection{Coordinate boxes}

\begin{defn}
We call an open set \( U\subset SM \) of diameter at most \( \delta  \)
\textbf{regularly
coordinated} if for all \( v,w\in U \) there are unique \( x,y \in U \) 
such that 

\[
x\in B^{u}_{\delta }(v),\: y\in B^{0}_{\delta }(x),\; w\in B^{s}_{\delta }(y).\]
 In other words, \( v \) can be joined to \( w \) by means of a unique
 short three-segment
path whose first segment is contained in \( W_{}^{u}(v), \) whose second
segment
is a piece of a flow line and whose third segment is contained in \(
W_{}^{s}(w). \) 
\end{defn}
\begin{prop}\label{prop:reg coord}
If \( v \) is regular then it has a regularly coordinated neighborhood.
\end{prop}
\begin{proof}
Some \( 4\delta  \)-neighborhood \( V \) of \( v \) is of rank one. Let 
\[
U=B^{s}_{\delta }(g^{(-\delta ,\delta )}B^{u}_{\delta }(v)).\]
 This is contained in \( V \) and hence of rank one. It is open since \( W^{0},
\)
\( W^{u} \) and \( W^{s} \) are transversal by
Lemma \ref{lem:reg iff transversal}. 

By construction, for any \( w\in V, \) there exists a pair \( (x,y) \) such
that 
\[
B_{\delta }^{u}(v)\ni x\in B_{\delta }^{0}(y),\, y\in B^{s}_{\delta }(w).\]
 Assume there is another pair \( (x',y') \) with this property. From 
\[
B_{\delta }^{u}(x)\ni v\in B_{\delta }^{u}(x')\]
 we deduce \( x\in B_{2\delta }^{u}(x') \), and from 
\[
B_{\delta }^{0}(x)\ni y\in B_{\delta }^{s}(w),\, 
w\in B_{\delta }^{s}(y'),\, y'\in B_{\delta
}^{0}(x')\]
 we deduce \( x\in B_{4\delta }^{0s}(x'). \) Hence \( x \) and \( x' \) are
 simultaneously positively and negatively asymptotic; therefore, they bound a
flat
strip. Since \( V \) is of rank one, there is no such strip of
nonzero width
in \( U \subset V\). Hence \( x \) and \( x' \) lie on the same geodesic. Since
\( x\in W^{u}(x') \),
these two points are identical.

The same argument with $u$ and $s$ exchanged shows that \( y=y'. \)
Hence the pair \( (x,y) \) is unique. 
\end{proof}
\subsection{The Busemann function and conformal densities}
\begin{defn}
Let \( b(.,q,\xi ) \) be the {\bf Busemann function} centered at \( \xi \in \tilde{M}
(\infty ) \)
and based at \( q\in \tilde{M} . \) It is given by 
\[
b(p,q,\xi ):=\lim _{p_{n}\rightarrow \xi }(d(q,p_{n})-d(p,p_{n}))\] 
for $p,q\in \tilde{M}$
and is independent of the sequence \( p_{n}\rightarrow \xi . \) 
\end{defn}
\begin{remark}
The function \textbf{\( b \)} satisfies \(b(p,q,\xi )=-b(q,p,\xi )\). Moreover,
\[
b(p,q,\xi )=\lim _{t\rightarrow \infty }(d(c_{p,\xi }(t),q)-t)\]
 where \( c_{p,\xi } \) is the geodesic parameterized by arclength
 with \( c_{p,\xi }(0)=p \) and \(
c_{p,\xi }(t)\rightarrow \xi  \)
as \( t\rightarrow \infty . \) 

For \( \xi \) and \(\, p \) fixed, we have 
\begin{eqnarray*}
b(p,p_{n},\xi )&\rightarrow& -\infty \textrm{ }\, \textrm{ for }\, \textrm{
}p_{n}\rightarrow \xi 
\\
b(p,p_{n},\xi )&\rightarrow&+ \infty \textrm{ }\, \textrm{ for }\,
\textrm{ }\lim _{n}p_{n}\in \tilde{M} (\infty )\setminus \{\xi \}.
\end{eqnarray*}
We use the sign convention where $b(p,q,\xi)$ is negative whenever $p,q,\xi$
lie on a geodesic in this particular order.
\end{remark}
\begin{defn}
\( (\mu_p)_{p\in\tilde{M}} \) is a \( h \)-dimensional 
{\bf Busemann density} (also called
{conformal density}) if the following are true:
\begin{itemize}
\item For all \( p\in \tilde{M} , \) \( \mu _{p} \) is a finite nonzero Borel measure
on
\( \tilde{M} (\infty ). \) 
\item \( \mu _{p} \) is equivariant under deck transformations, i.e., for all
\( \gamma \in \pi _{1}(M) \)
and all measurable \( S\subset \tilde{M} (\infty ) \) we have 
\[
\mu _{\gamma p}(\gamma S)=\mu _{p}(S).
\] 
\item When changing the base point of \( \mu _{p}, \) the density transforms as
follows:
\[
\frac{d\mu _{p}}{d\mu _{q}}(\xi )=e^{-hb(q,p,\xi )}.
\]
\end{itemize}
\end{defn}
In the case of nonpositive curvature, Knieper has shown in \cite{knie
ann} that \( \mu _{p} \) is unique up to a multiplicative 
factor and that it can be obtained by the Patterson-Sullivan
construction.
\section {The measure of maximal entropy}
In section \ref{lastsection} 
 we will use the fact that if $m$ is the
measure of maximal
entropy then it gives rise to conditional measures $m^u$, $m^{0u}$,
$m^s$ and $m^{0s}$  on unstable, weakly unstable, stable and weakly
stable leaves
which have the property that 
the
measures $m^{0u}$ and $m^u$ expand uniformly with $t$
and that  $m^{s}$ and $m^{0s}$ contract uniformly with $t$.

\begin{remark} In \cite{Gun} we present an alternative and more general
 construction of the measure of
maximal entropy in nonpositive curvature and rank one which 
follows the principle of Margulis'
 construction. Using that construction, the uniform
expansion/contraction properties shown here are already a
straightforward consequence of
the construction. 
Also, that construction works for non-geodesic
flows satisfying
suitable cone 
conditions (see \cite{Kat2} for these).
 On the other hand, Knieper's approach, which substantially
 requires properties of rank one nonpositively curved manifolds, 
is shorter and therefore is the one we use in this article.

\end{remark}

First we give Knieper's definition of the measure of maximal
entropy (\cite{knie ann}):
\begin{defn}\label{def:knieper def}
Let $(\mu_p)_{p\in\tilde{M}}$ be a Busemann density. Let $$\Pi:S\tilde M\to
\tilde{M}(\infty)\times \tilde{M}(\infty),\quad
\Pi(v):=(v_\infty,v_{-\infty})$$ be the
projection of a vector to both endpoints 
$v_{\pm\infty}=\lim_{t\to\pm\infty}\pi g^tv$ of the corresponding geodesic.
Then the measure of maximal entropy of a set 
$A\subset S\tilde{M}$ (we can without loss of generality assume $A$ to be
regular) is given
by
\[m(A):=\int\limits_{\xi,\eta\in
\tilde{M}(\infty),\ \xi\neq\eta}\len(A\cap
\Pi^{-1}(\xi,\eta))e^{-h(b(p,q,\xi)+b(p,q,\eta))}d\mu_p(\xi)d\mu_p(\eta),\]
where $q\in\pi\Pi^{-1}(\xi,\eta)$ and $p\in\tilde{M}$ is arbitrary.
\end{defn}
Here $\len$ is the length of the geodesic segment. Saying that
$\Pi^{-1}(\xi,\eta)$ is a geodesic
already is a slight simplification, but a fully justified one since we need to
deal only with the regular set.

\subsection{Discussion of the conditionals}
Given a vector $v\in S\tilde M$ with base point $p$, we want to put a
conditional measure $m^u$ on the stable horosphere $b(p,.,\xi)^{-1}(0)$
given by $v$ and 
centered at $\xi:=v_\infty$
(or on $W^s(v)$, which is the unit normal bundle of
$b(p,.,\xi)^{-1}(0)$). This conditional is determined by a
multiplier with respect to some given measure on this horosphere.
Note that the set of points $q$ on the horosphere is parameterized by
the set $\tilde M(\infty)\setminus\{\xi\}$ via projection from $\xi$
 into the boundary at
infinity, hence the multiplier
depends on
$\eta:=\text{proj}_\xi(q)$, i.e. is proportional to $d\mu_x(\eta)$ for some $x$. 
The  canonical choice for $x$ is $p$.
Clearly the whole horosphere has infinite $m^u$-measure, but $\mu_x$ is
finite for any $x$. Thus the multiplier of $d\mu_p$ has to have a singularity,
and
this has to happen at $\eta=\xi$ since any neighborhood of $\xi$ is the
projection of 
the complement of a
compact piece of the horosphere. The term 
$e^{-hb(p,q,\eta)}$ has the right singularity (note that $\eta\to\xi$
means $q\to\xi$), and  by the basic properties of the Busemann 
function the term $e^{-hb(p,q,\eta)}$ then
converges to infinity. 
Therefore we  investigate
${m}_p(q):=e^{-hb(p,q,\eta)}d\mu_p(\eta).$ First
 we prove that this is
indeed the stable conditional measure for $dm^s.$
We will parameterize $dm$ by vectors instead of their base points.

\begin{defn}\label{def cond} For $v,w\in S\tilde M$, let
\begin{eqnarray*}
dm_v^u(w)&:=&e^{-hb(\pi v,\pi
w,w_\infty)}\cdot d\mu_{\pi v}(w_\infty),\\
dm_v^s(w)&:=&e^{-hb(\pi v,\pi
w,w_{-\infty})}\cdot d\mu_{\pi v}(w_{-\infty}).
\end{eqnarray*}
\end{defn}

\begin{prop}\label{prop:condKn} $dm_v^s, dm_v^u$ and $dt$ are the stable, unstable
and center 
conditionals of the measure of maximal entropy. 
\end{prop}
\begin{proof}
Observe that \begin{eqnarray*}\label{newformula}
dt\,dm_v^u(w)\,dm_v^s(w)&=&dt\,e^{-h(b(\pi v,\pi w,w_\infty)+b(\pi v,\pi
w,w_{-\infty}))}\\&&\cdot\, d\mu_{\pi v}(w_\infty)d\mu_{\pi v}(w_{-\infty})
\\&=&dt\,e^{-h(b(p,q,\xi)+b(p,q,\eta))}d\mu_{p}(\xi)d\mu_{p}({\eta})
\ =:E
\end{eqnarray*}
with $p:=\pi v,\ q:=\pi w,\ \xi:=w_\infty, \ \eta:=w_{-\infty}$. This
formula already agrees with the formula in Definition \ref{def:knieper def},
although the
meaning of the parameters does not
necessarily do so: In Definition \ref{def:knieper def},
$p$ and to some extend $q$ are arbitrary in $\tilde M$, while in the 
formula for $E$ they 
are fixed. Thus we need to show that if we change them within the range
allowed in Definition \ref{def:knieper def}, the value of $E$ does not
change.
\begin{lemma}
The term  $E$ does not change if 
$q$ is replaced by any point in $\tilde M$ on the geodesic
$c_{\eta\xi}$ from  $\eta$ to $\xi$
and 
$p$ by an
arbitrary point in $\tilde M$.
\end{lemma}
\begin{proof}
First we show that $q$ can be allowed to be anywhere on $c_{\eta\xi}$:
Parameterize $c_{\eta\xi} $ by arclength with arbitrary parameter shift in the
direction 
from $\eta$ to $\xi$. Replacing
$q=c_{\eta\xi}(s)$ by $q'=c_{\eta\xi}(s')$ changes $b(p,q,\xi)$ to
$b(p,q',\xi)=b(p,q,\xi)-(s'-s)$ since we move the distance $s'-s$ closer
to $\xi$. It also changes $b(p,q,\eta)$ to
$b(p,q',\eta)=b(p,q,\eta)+(s'-s)$ since we move the distance $s'-s$ away
from 
to $\eta$.
Thus  $E$ does not change under such a translation of $q$.

Now fix $q$ anywhere on $c_{\eta\xi}$ and replace $p$ by some 
arbitrary $p'\in \tilde M$. Note that  
\begin{eqnarray*}
d\mu_{p'}(\xi)&=&e^{hb(p',p,\xi)}d\mu_p(\xi),\\
b(p',q,\xi)&=&b(p,q,\xi)+b(p',p,\xi),
\end{eqnarray*}
which of course also holds with $\xi$ replaced by $\eta$.
Thus
$$e^{-h(b(p',q,\xi)+b(p',q,\eta))}d\mu_{p'}(\xi)d\mu_{p'}({\eta})
=e^{-h(b(p,q,\xi)+b(p,q,\eta))}d\mu_{p}(\xi)d\mu_{p}({\eta}).$$ Hence 
  $E$ also does not change if $p$ is changed to any arbitrary
  point.
\end{proof}
This also concludes the proof of Proposition \ref{prop:condKn}.
\end{proof}
\subsection{Proof of uniform expansion/contraction of the conditionals}

Let $w_t$ denote $g^tw$.
\begin{thm}[Uniform expansion/contraction of the conditionals]
\label{thm:unif-exp-of-cond}
For all $t\in\mathbb{R}$ and all $v,w\in S\tilde M$ we have
$$dm_v^u(w_t)=e^{ht}\cdot dm_v^u(w),$$
$$dm_v^s(w_t)=e^{-ht}\cdot dm_v^s(w).$$
The same uniform expansion holds with $dm^u$ replaced by $dm^{0u}=dm^u dt$ and the same
uniform 
contraction with  $dm^s$ replaced by $dm^{0s}=dm^s dt$.
\end{thm}
\begin{proof}
\begin{eqnarray*}
dm_v^s(w_t)
&=&e^{-hb(\pi v,\pi w_t,w_{-\infty})}d\mu_{\pi v}(w_{-\infty})
\\&=&e^{-h(b(\pi v,\pi w,w_{-\infty})+b(\pi w,\pi w_t,w_{-\infty}))}d\mu_{\pi
v}(w_{-\infty})
\\&=&e^{-hb(\pi v,\pi w,w_{-\infty})-ht}d\mu_{\pi v}(w_{-\infty})
\\&=&e^{-ht}\cdot e^{-hb(\pi v,\pi w,w_{-\infty})}d\mu_{\pi v}(w_{-\infty})
\\&=&e^{-ht}\cdot dm_v^s(w).
\end{eqnarray*}
Similarly, the equality 
$b(\pi v,\pi w_t, w_{+\infty}) = b(\pi v, \pi w,w_{+\infty})+t$ 
 yields
$$dm_v^u(w_t)=e^{ht}\cdot dm_v^u(w).$$
This shows the desired uniform expansion of $m^u$ and the uniform
contraction of $m^{s}$. 
From this we also immediately see the uniform expansion of $m^{0u}$ and
the uniform contraction of $m^{0s}$ since $dt$ is evidently invariant
under $g^t$.
\end{proof}

\subsection{Proof of holonomy invariance of  the conditionals}

Another important property of the conditional measures on the leaves is
holonomy invariance. 
We formulate holonomy invariance on infinitesimal unstable pieces here, but of
course this is equivalent to holonomy invariance that deals with pieces of
leaves of
(small) positive size. 

We consider positively asymptotic vectors $w,w'$ and
calculate the infinitesimal $0u$-measure on corresponding leaves. We let
$v,v'$ be some (arbitrary) base points used as parameters for the pieces of
leaves, so that $w$ lies in the same $0u$-leaf of $v$ and similarly $w'$ in
that of $v'$. The factor $dt$ is evidently invariant, so we do not have to
mention it any further.

\begin{thm}[Holonomy invariance of  the conditionals of the measure of maximal
entropy]\label{thm:holonomyinvariance}
$$dm_v^u(w)=dm_{v'}^u(w')$$ whenever $v'\in W^s(v),$ $w'\in W^s(w),$ $w\in
W^{0u}(v)$ and $w'\in W^{0u}(v').$ 
In that case also $dm_v^{0u}(w)=dm_{v'}^{0u}(w')$ holds.

Similarly,
$$dm_v^s(w)=dm_{v'}^s(w')$$ whenever $v'\in W^{u}(v),$ $w'\in W^{u}(w),$ $w\in
W^{0s}(v)$ and $w'\in W^{0s}(v'),$ and in that case also
$dm_v^{0s}(w)=dm_{v'}^{0s}(w')$ holds. 
\end{thm}
\begin{proof}
Note that the equation $w'\in W^s(w)$ is equivalent to the two equations
$$w'_\infty=w_\infty,$$
$$b(\pi w,\pi w',w_\infty)=0.$$ 
The latter equation is equivalent to $b(p,\pi w,w_\infty)=b(p,\pi w',w_\infty)$
for all $p\in\tilde M.$ Thus clearly
\begin{eqnarray*} dm_{v'}^u(w')
&=& e^{-hb(\pi v',\pi w',w_\infty)}d\mu_{\pi v'}(w'_\infty)
\\&=& e^{-hb(\pi v',\pi w',w_\infty)}d\mu_{\pi v'}(w_\infty).
\end{eqnarray*}
Now \begin{eqnarray*}b(\pi v',\pi w',w_\infty)
&=& b(\pi v',\pi v,w_\infty)+b(\pi v,\pi w',w_\infty)
\\&=& b(\pi v',\pi v,w_\infty)+b(\pi v,\pi w,w_\infty)
 \end{eqnarray*}
and $d\mu_{\pi v'}(w_\infty)=e^{-hb(\pi v,\pi v',w_\infty)}d\mu_{\pi
v}(w_\infty).$ Thus
\begin{eqnarray*} dm_{v'}^u(w')
&=& e^{-h(b(\pi v',\pi w',w_\infty)+b(\pi v,\pi v',w_\infty))}d\mu_{\pi
v}(w_\infty)
\\&=& e^{-h(b(\pi v',\pi v,w_\infty)+b(\pi v,\pi w,w_\infty)+b(\pi v,\pi
v',w_\infty))}d\mu_{\pi v}(w_\infty)
\\&=& e^{-hb(\pi v,\pi w,w_\infty)}d\mu_{\pi v}(w_\infty)
\\&=& dm^u_{ v}(w).
\end{eqnarray*}
The proof for $dm^s$ is analogous.
\end{proof}
Note that $m^{0u}$ is invariant under holonomy along $s$-fibers and $m^{0s}$ 
under holonomy along $u$-fibers, but  $m^u$ is not invariant under holonomy
along 
$0s$-fibers and $m^s$  not invariant under holonomy along $0u$-fibers
due to expansion (resp. contraction) in the flow direction. 
\section{Counting closed geodesics}
\label{lastsection}
In this final section we count the periodic geodesics on \( M. \)
The method used here is a generalization of the method which, for the special
case of negative curvature, was outlined in \cite{Tol} and provided with
more detail in \cite{kh}. Margulis (\cite{Mar1}, \cite{Mar4}) is the originator
of that method, although the presentation in this article looks quite different.
\begin{defn}
\label{defrelations}
Let $f=f(t,\varepsilon),\ g=g(t,\varepsilon):[0,\infty)\times (0,1)\to
(0,\infty)$ 
 be expressions depending on \( t \) and \( \varepsilon
. \) 
We are interested in the behavior for $t$ large and $\varepsilon>0$ small.

Write 
\[
f\sim g\]
if for all $\alpha>0$ there exists $\varepsilon_0 >0 $ such that for all
$\varepsilon\in(0,\varepsilon_0)$ there exists $t_0\in(0,\infty)$ such that for all
$t>t_0$  we have
\[\left|\ln\frac{f(t,\varepsilon)}{g(t,\varepsilon)}\right| <\alpha.\]

Write 
\[
f\bowtie g\]
if there exists $K\in\mathbb{R}$, $\varepsilon_0>0$ such that for all 
$\varepsilon\in(0,\varepsilon_0)$ there exists $t_0\in(0,\infty)$ such that for 
all $t>t_0$ we have
\[\left|\ln\frac{f(t,\varepsilon)}{g(t,\varepsilon)}\right| <K\varepsilon.\]

We write
\[
f\cong g\]
 if there exists \( f' \) with \( f\sim f'\bowtie g, \) i.e. if 
there exists $K\in\mathbb{R}$ so that for all $\alpha>0$ there exists $\varepsilon_0>0$
such that for all $\varepsilon\in(0,\varepsilon_0)$ there exists $t_0>0$ 
such that for all $t>t_0$ we have
\[\left|\ln\frac{f(t,\varepsilon)}{g(t,\varepsilon)}\right| <K\varepsilon+\alpha.\]
\end{defn} 
 Thus the relation
``$\cong$'' is
implied by both ``$\sim$'' and ``$\bowtie$'', 
which are the special cases $K=0$ and $\alpha=0$, respectively;
but the relations
 ``$\sim$'' and ``$\bowtie$'' are independent.
\begin{remark}Similarly, in the definition of ``$\sim$'' and ``$\cong$'', the variable
$t_0$ may depend on $\varepsilon$, i.e. the convergence in $t$ does not have to be 
uniform with respect to $\varepsilon$. All arguments in the rest of this article
work without requiring this uniformity.
\end{remark}
\begin{remark}Obviously these relations are also well-defined if the domain of
the functions $f,g$
is $[T_0,\infty)\times(0,\gamma)$ for some $T_0,\gamma>0$.
\end{remark}
\begin{lemma}
\label{lem:relationsarerelations}
The relations ``$\sim$'', ``$\bowtie$'' and ``$\cong$''
 are equivalence relations. 
\end{lemma}
\begin{proof}
It suffices to consider 
``$\cong$'' since the others are special cases of it. 
Reflexivity and symmetry are trivial.
If the functions $f_1,f_2,f_3:[0,\infty)\times (0,\infty)\to
(0,\infty)$ satisfy $f_1\cong f_2\cong f_3$, i.e., for $i=1,2$ we have
\(\exists K_i\ \forall\alpha>0\ \exists \varepsilon_{0,i}>0
\ \forall \varepsilon\in(0,\varepsilon_{0,i})\ \exists t_{0,i}>0\ \forall t>t_{0,i}:
\left|\ln (f_i(t,\varepsilon)/f_{i+1}(t,\varepsilon))\right| 
<K_i\varepsilon+\alpha/2\), then clearly 
\(\exists K_3\ \forall\alpha>0\ \exists \varepsilon_{0,3}>0
\ \forall \varepsilon\in(0,\varepsilon_{0,3})\ \exists t_{0,3}>0\ \forall
t>t_{0,3}:\)
\[\left|\ln({f_1(t,\varepsilon)}/{f_{3}(t,\varepsilon)})\right| 
<K_3\varepsilon+\alpha,\] namely $K_3:=K_1+K_2$,
$\varepsilon_{0,3}:=\min(\varepsilon_{0,1}(\alpha/2),\varepsilon_{0,2}(\alpha/2))$ and 
$t_{0,3}:=\max(t_{0,1}(\varepsilon_{0,3}),t_{0,2}(\varepsilon_{0,3})).$ 
This shows $f_1\cong f_3$.
\end{proof}
\begin{remark}
In Definition \ref{defrelations},
we could have written $\left|\frac{f(t,\varepsilon)}{g(t,\varepsilon)}-1\right|$
instead of $\left|\ln\frac{f(t,\varepsilon)}{g(t,\varepsilon)}\right|$.
This would be equivalent to our definition
since the terms
$\ln x$ and $|x -1|$ differ by at most a factor 2 (indeed any $a>1$) 
for all $x$ close enough to 1. The advantage of our notation is that
multiple estimates can easily be transitively combined, 
as seen in the proof of Lemma 
\ref{lem:relationsarerelations}. Also, our notation is symmetric in $f,g$.
\end{remark}
\subsection{The flow cube}
Fix any \( v_0\in \textrm{{\bf{Reg}}}. \) 
Choose sufficiently small \( \varepsilon >0 \) and
\( \delta >0 \)
such that $4\delta<2\varepsilon<\text{inj}(M)$ (the
injectivity radius of $M$), such that 
\( B_{4\varepsilon }(v_0)\subset \textrm{{\bf{Reg}}} \), and such that 
further requirements on the smallness of these which we will mention later are
satisfied. 
\begin{defn}
Let the \textbf{flow cube} be 
\(
A:=\overline{{B}^{s}}(g^{[0,\varepsilon ]}(\overline{{B}^{u}_{\delta
}}(v_0)))\subset \textrm{{\bf{Reg}}}.\)
Here $\overline{{B}^{u}_{\delta}}(v_0)$ is
the closed unstable ball of radius
$\delta$ around $v_0$. We choose 
$\overline{B^s}=\overline{B^s}(v)$ as
the closure of an open set contained in
the closed stable ball of radius $\delta$ around $v\in 
g^{[0,\varepsilon ]}(\overline{{B}^{u}_{\delta
}}(v_0))$; 
this set, which depends on $v$, 
can be chosen in such a way that it contains $v$ 
and that $A$ has
the following local product structure:
For all \( w,w'\in A \) there exists a unique \( \beta \in [-\varepsilon ,\, \varepsilon
] \)
such that 
\[
\overline{B^{s}}(w)\cap \overline{B^{u}_{2\delta }}(g^{\beta }w')\]
 is nonempty, and in that case it is exactly one point. 
This is the local product structure in {\bf{Reg}} 
described in Proposition \ref{prop:reg coord}.
We call $\overline{B^s}(v)$ the {\bf{stable fiber}}
(or stable ball) in $A$ containing $v$.
\begin{figure}[h]
\begin{center}
\footnotesize
\psfrag{d}{$\delta$}
\psfrag{z}{$v_0$}
\psfrag{e}{$\varepsilon$}
\includegraphics{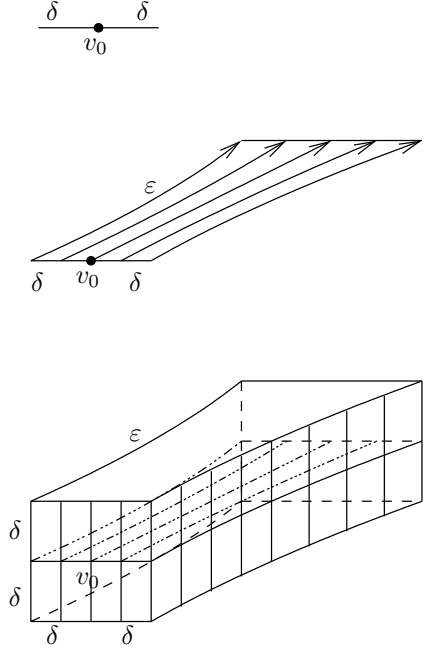}
\label{flowcube}\caption{The flow cube \protect\( A\protect \): an unstable
neighborhood 
of \protect \(v_0\) (top) is iterated (center) and a stable neighborhood of
that is formed (bottom).}
\end{center}
\end{figure}
\normalsize
\end{defn}
In the following arguments, the cube \( A \) will first be fixed. In particular, 
\( \varepsilon 
\)
and \( \delta  \) are considered fixed (although subject to restrictions on their
size). Then we make asymptotics certain numbers depending on $t$ and $A$ as
$t\to\infty$ (while $A$, hence $\varepsilon$, 
is fixed). Afterwards we will consider what happens to those
asymptotics when 
\( \varepsilon \to
0. \)

\begin{defn}
Let the \textbf{depth} \( \tau:A\to[0,\varepsilon]  \) be defined  by 
\[
 v\in \overline{B^{s}}(g^{\tau (v)}\overline{B^{u}_{\delta }}(v_0)).\]
\end{defn}
\begin{lemma}
For all \( v\in A,\, w\in \overline{B^{u}_{2\delta }}(v)\cap A \) it is true that
\[
|\tau (w)-\tau (v)|<\varepsilon ^{2}/2.
\]
\end{lemma}
\begin{proof}
The foliation \( W^{u} \) is uniformly continuous by Theorem \ref{thm:unif cts}
and compactness of $SM$, 
and without loss of generality $\delta$ was chosen small enough.
\end{proof}
\begin{lemma}[Stable fiber contraction]\label{lem:sizeofstablefibers}
There is a function \( \sigma =\sigma (t) \) such that 
\[
m^{s}(g^{t}B^{s}(v))\bowtie \sigma (t)\]
 for all \( v\in A \).
In particular, for all \( v,w \) in \( A \) we have
\[
m^{s}(g^{t}B^{s}(v))\bowtie m^{s}(g^{t}B^{s}(w)).\]
Moreover, the constants in the relation ``$\bowtie$'' can be chosen independent
 of $v,w$, i.e.,  there exists $K>0$, $\varepsilon_0>0$ such that for all 
$0<\varepsilon<\varepsilon_0$ and all flow cubes $A=A(\varepsilon)$
there exists $t_0$ such that for 
all $t>t_0$ and all $v,w\in A$ we have
\[\left|\ln\frac{m^s(g^tB^s(v))}{\sigma(t)}\right| <K\varepsilon \text{ \ and \ }
\left|\ln\frac{m^s(g^tB^s(v))}{m^s(g^tB^s(w))}\right| <K\varepsilon.\]
\end{lemma}
\begin{proof}
First we show the second claim. Observe that
for any $a\in(0,\text{inj}(M)/2)$ (with $a$ independent of $\varepsilon$),
the set $g^{[0,a]}B^s(v)$ is
$u$-holonomic to a subset $S$ of $ g^{[-2\varepsilon,a+2\varepsilon]}B^s(w).$ 
Thus	
$$\frac{m^{s}(B^{s}(v))}{
m^{s}(B^{s}(w))}
=\frac{m^{0s}(g^{[0,a]}B^{s}(v))}{
         m^{0s}(g^{[0,a]}B^{s}(w))}
=\frac{m^{0s}(S)}{
         m^{0s}(g^{[0,a]}B^{s}(w))}
$$$$
\leq\frac{m^{0s}(
g^{[-2\varepsilon,a+2\varepsilon]}B^s(w))}{
         m^{0s}(g^{[0,a]}B^{s}(w))}
=\frac{
\int_{-2\varepsilon}^{a+2\varepsilon}e^{-ht}dt}{
\int_{0}^{a}e^{-ht}dt}\bowtie 1,$$ 
since the  quotient of the integrals can be bounded by $1+K\varepsilon$.
The inequality is symmetric in $v$ and $w$, proving equality.
Hence  \( m^{s}(B^{s}(v))\bowtie m^{s}(B^{s}(w)) \), i.e. 
$\exists K>0,\ \varepsilon_0>0\ \forall \varepsilon\in(0,\varepsilon_0)\ \exists
t_0\ \forall t>t_0\ \forall v,w\in A:$
$$\left|\ln\frac{m^s(B^s(v))}{m^s(B^s(w))}\right| < K\varepsilon.$$
Using uniform contraction on $s$-fibers (Theorem \ref{thm:unif-exp-of-cond})
gives 
$$\left|\ln\frac{m^s(g^tB^s(v))}{m^s(g^tB^s(w))}\right| < K\varepsilon,$$
i.e., $m^{s}(g^{t}B^{s}(v)) = m^{s}(g^{t}B^{s}(w))$,
 showing the second claim.
It immediately follows that we can  define 
\[ \sigma (t):=m^{s}(g^{t}B^{s}(v)) \]
for some arbitrary \( v\in A \), and
this definition does not depend on \( v \)
(up to \( \bowtie  \)-equivalence).
The constant $K$ in $\bowtie$ is independent of $v$. This also shows the first claim. 
\end{proof}
\begin{remark}\label{rem:sigma asym}
Uniform contraction (Theorem \ref{thm:unif-exp-of-cond}) then shows that  $\sigma(t)=
\textrm{const}\cdot e^{-ht} .$
\end{remark}
\subsection{Expansion at the boundary}
\begin{defn}
For the cube \( A \) as above, we call 
\begin{eqnarray*}
 &\partial^uA:=\overline{{B}^{s}}(g^{[0,\varepsilon ]}(\partial
B^{u}_{\delta 
}(v_0))) & \text{the \textbf{unstable end} of the cube,}\\
& \partial^sA:={(\partial{B}^{s})}(g^{[0,\varepsilon ]}(
\overline{B^{u}_{\delta 
}}(v_0))) & \text{the \textbf{stable end},}\\
&\partial_0A:= \overline{{B}^{s}
}(\overline{{B}^{u}_{\delta }}(v_0)) & \text{the \textbf{back end} and} \\
&  \partial_\varepsilon A:=\overline{{B}^{s}
}(g^{\varepsilon
}(\overline{{B}^{u}_{\delta }}(v_0))) & \text{the \textbf{front end} of the
cube.} 
\end{eqnarray*}
For \( v\in A \) define 
\[
s(v):=\sup {\{r:B^{u}_{r}(v)\subset A\}}
\]
 to be the \textbf{distance to the unstable end} of the flow cube.
\end{defn}
The stable and the unstable end are
topologically the product of an interval, a $k$-ball and a $(k-1)$-sphere,
where $k=\dim W^u(v)=\dim W^s(v)=(\dim SM -1)/2;$ hence they are connected iff
$k\neq 1$, i.e. iff $M$ is not a surface.
\begin{lemma}[Expansion of the distance to the unstable end]\label{lem:side-to-side}
There exists a monotone positive function \( S:\mathbb {R}\rightarrow \mathbb
{R} \)
satisfying \( S(t)\rightarrow 0 \) as \( t\rightarrow \infty  \) and such
that if \( s(v)>S(t) \) for an element \( v\in A \) which satisfies \(
g^{t}v\in A \) then
\[\overline{B^{u}_{2\delta }}(g^tv)\cap A\subset
g^{t}\overline{B^{u}_{S(t)}}(v).
\]
\end{lemma}
That means that 
if a point $v$ is more than $S(t)$ away
from 
the 
unstable end of the cube then the 
the image under $g^t$ of a small
$u$-disc (of size $>S(t)$) around $v$
has the property that
its unstable end is completely
outside $A$.
\begin{proof}
By nonpositivity of the curvature, \(
B^{u}_{\delta } \) noncontracts, i.e., for all $p,q\in B_\delta^u$ the function
$t\mapsto d(g^t p,g^t q) $ is nondecreasing.
This is true even infinitesimally, i.e.
for unstable Jacobi fields. 
By convexity of Jacobi fields and rank 1, such distances
also cannot stay bounded.
Hence 
the radius of the largest $u$-ball
contained in $g^tB^u_\delta$ becomes 
unbounded for $t\to\infty.$

Hence for all $\gamma>0$ we can find $T_\gamma<\infty$ such that 
\begin{equation}\label{gT supset}
 g^{T_\gamma}B^u_{\gamma}(v)\supset B^u_{2\delta}(g^{T_\gamma}(v)).
 \end{equation}
By compactness of $A$, this choice of $T_\gamma$
 can be made independently of $v\in A$.
Without loss of generality $T_\gamma $ is a strictly
decreasing function of $\gamma$. 
Choose a function $S:[0,\infty)\to(0,\infty)$ so that $S(t)\leq\gamma$ for $t>T_\gamma.$ 
E.g., choose $S(.)=T_{.}^{-1}$, i.e. $T_{S(t)}=t$ for $t\geq 0$.
$S$ can be chosen decreasing since $T_.$ can be.
Therefore, given  $v\in A$, if
$t>T_{s(v)}$ then $s(v)>S(t)$, and
thus  equation (\ref{gT supset})
shows the claim.
\end{proof}
\begin{remark}
The convergence of \( S \) to zero in the previous Lemma is not
necessarily
exponential, as opposed to the case where the curvature of $M$
 is negative (i.e.  the 
uniformly hyperbolic case). However, we do not
need this property of exponential convergence.

If the smallest such \( S \) would not converge to zero, it would require the
existence of a flat strip of width 
\(
\liminf _{t\rightarrow \infty }S(t)=\lim _{t\rightarrow \infty
}S(t),\)
 which would intersect \( {A}. \) Since a
 neighborhood of \( A \)
is regular,
this cannot happen.
\end{remark}
\subsection{Intersection components and orbit segments}
\begin{defn}
Let \( A_{t}' \) be the set of \( v\in A \) with \( s(v)\geq S(t) \) and 
\( \tau (v)\in [\varepsilon ^{2},\varepsilon -\varepsilon ^{2}] \). 
Thus \( A_{t}' \) is the set \( A \) with a small neighborhood of the unstable
end
and of the front end and back end removed.
\end{defn}
\begin{defn}
Let \( \Phi _{t} \) be the set of all \textbf{full components of intersection}
at time $t$:
If $I\textrm{ is a connected component of }A'_{t}\cap
g^{t}(A'_{t})$ then define
\[
\phi _{t}^I:=g^{[-\varepsilon ,\varepsilon
]}(I)\cap A\cap g^{t}(A) ,
\] 
\[
\Phi_t:=\{\phi_t^I:I\textrm{ is a connected component of }A'_{t}\cap
g^{t}(A'_{t})\}.
\]

Let $N(A,t):=\#\Phi_t$
be the number of  elements of $\Phi_t$.

We call the set \(g^{[-\varepsilon ,\, \varepsilon ]}v\cap A\) the  
\textbf{geometric orbit segment} of length \(  \varepsilon  \) in $A$  through
\( v. \)   Similarly we
speak about the 
orbit segment of length \(  \varepsilon-2\varepsilon^2  \)  in $A'_t$.

Let $\Phi_t^s:=\{\phi_t^I\in\Phi_t:\phi_t^I\text{ intersects }\partial^s A'_{t}\}$.
\end{defn}
\begin{lemma}\label{lem:segtocomp}
For every geometric orbit segment of length \( \varepsilon -2\varepsilon ^{2} \) 
in \( A'_{t} \) that
belongs
to a periodic orbit of period in 
\( [t-\varepsilon +2\varepsilon ^{2},t+\varepsilon -2\varepsilon ^{2}] \)
there exists a unique $\phi_t^I\in\Phi_t$ through which the geometric orbit segment
passes.
\end{lemma}
\begin{proof}
Existence: If $g^Lo=o$ for an orbit segment $o $ of length 
\( \varepsilon -2\varepsilon ^{2} \)
of \( A'_{t} \) that belongs
to a periodic orbit of period  
\(L\in [t-\varepsilon +2\varepsilon ^{2},t+\varepsilon -2\varepsilon ^{2}] \)
then $o$ also intersects $g^tA'_t$, hence some component of
$A'_t\cap g^tA'_t.$

Uniqueness: Assume that $o$ passes through
$\phi_t^I,\phi_t^J\in\Phi_t,\,$
i.e. $p=o(a)\in \phi_t^I,\ q=o(b)\in \phi_t^J$ for $ |b-a|\leq\varepsilon$.
Then $o$ passes through $I,J$ (the connected
components corresponding to $\phi_t^I,\phi_t^J$) respectively. 
Without loss of generality, $p,q\in A'_t$.
Since $g^tA_t'$ is pathwise connected,
there is a path $c$ in $g^tA_t'$ from $p$ to
$q$. 
Using the local product structure,
we can assume that
$c$ consists of a segment in $W^u$, followed
by a segment in $W^0$, followed by a
segment in $W^s$.
By applying $g^{-t}$, we get a path $g^{-t}\circ c$ 
in $A_t'$ from $o(a-t)\in A_t'$ to
$o(b-t)\in
A_t'$.
The local product structure in $A_t'$
 and the fact that 
distances along unstable fibers become $>2\varepsilon$ for $t\to\infty$  
whereas 
distances along stable fibers become $>2\varepsilon$ for $t\to - \infty$
(see Lemma \ref{lem:side-to-side}) 
show that the $u$-segment and the
$s$-segment of $g^{-t}\circ c$
have length 0. 
Therefore $g^{-t}\circ c$ and hence $c$ is an
orbit segment. This means that $c$ lies in
$A_t'$ and in $g^tA_t'$. Hence $p$ and $q$
lie in the same component, i.e.
$\phi_t^I=\phi_t^J$.
\end{proof}
In the other direction, we have the following Lemma:
\begin{lemma}\label{lem:comptoseg}
For every \( \phi_t^I\in\Phi_t \setminus \Phi_t^s\) 
there exists a unique periodic orbit with period in 
\( [t-\varepsilon ,\, t+\varepsilon ] \) 
and a unique geometric orbit segment on that orbit passing through $\phi_t^I$.
\end{lemma}
In other words, up to a small error, intersection components correspond to
periodic orbits, and of all orbit segments that belong to such a
periodic
orbit, just one orbit segment goes through any particular full component of
intersection. 
\begin{proof}
Choose  $\phi_t^I$. It suffices to consider the case
$t\geq 0$.
 Since $A_t'\subset A$ has rank one, it follows that
for every $v\in A_t'$ any nonzero stable 
Jacobi field along $c_v$ is strictly 
decreasing in length,
and any nonzero unstable Jacobi field is strictly increasing in length. 
Since the set
of stable (resp. unstable) Jacobi fields is linearly isomorphic to $E^s$
(resp. $E^u$) via $(d\pi,\mathcal{K})^{-1}$, it follows that for all $v\in
A'_t\cap g^tA'_t$:
$$|dg^t\xi|<|\xi|\quad \forall\xi\in E^s(v)\setminus\{0\},$$
$$|dg^{-t}\xi|<|\xi|\quad \forall\xi\in E^u(v)\setminus\{0\}.$$
By compactness of $A_t'$ and hence of $\phi_t^I$ there exists $c<1$ 
such that for all $v\in A'_t\cap g^tA'_t$:
$$|dg^t\xi|<c|\xi|\quad\forall\xi\in E^s(v)\setminus\{0\},$$
$$|dg^{-t}\xi|<c|\xi|\quad\forall\xi\in E^u(v)\setminus\{0\}.$$
Hence  $g^t$ 
restricted to $\phi_t^I$ is (apart from the flow direction) hyperbolic.
By the assumption that $\phi_t^I\not\in\Phi_t^s$, stable fibers are mapped to stable
fibers that do not intersect the stable end of the flow cube.
Thus the first return map on a transversal to the flow is hyperbolic. 
Hence it has a unique fixed point.

Therefore there exists a unique periodic orbit through $\phi_t^I$.  
Two geometrically different (hence disjoint) 
orbit segments would give rise to
two different fixed points.
Hence the geometric segment on the periodic
orbit is also unique.
\end{proof}
\subsection{Intersection thickness}
\begin{defn}
Define the \textbf{intersection thickness} (or intersection length)
 $\theta:\Phi_t\to[0,\varepsilon]$ 
by 
\[
\theta (\phi_t^I):=\varepsilon -
\sup\left\{\tau (v):v\in g^{t}\left(\bigcup _{w\in A,\,
g^{t}w\in I}g^{[-\varepsilon,\varepsilon]}w
\cap \partial _{0}A\right)\right\}\]
 for such \( \phi_t^I \) which intersect 
 \( \partial _{\varepsilon}A\)
(the front end of $A$) and 
\[
\theta (\phi_t^I):=\inf \left\{\tau (v):v\in g^{t}\left(\bigcup _{w\in A,\,
g^{t}w\in I}g^{[-\varepsilon,\varepsilon]}w
\cap \partial _{\varepsilon }A\right)\right\}\]
for such \( \phi_t^I \) which intersect 
\( \partial _{0}A \) (the back end of $A$).

\end{defn}
\begin{lemma}
[The average thickness is asymptotically half that of the flow box]
\label{lem:avgthickness}
 $$
\frac{1}{N(A,t)}\sum_{ \phi_t^I\in\Phi_t}
\theta(\phi_t^I)\cong\frac{\varepsilon}{2}.$$
In other words, 
\(\exists K<\infty\ \forall\alpha>0\ \exists \varepsilon_{0}>0
\ \forall \varepsilon\in(0,\varepsilon_{0}),\ A=A(\varepsilon)
\ \exists t_{0}>0\ \forall t>t_0:$
$\frac{1}{N(A,t)}\left|\ln ({2\sum_{ \phi_t^I\in\Phi_t}{\theta(\phi_t^I)}}/\varepsilon)
\right| 
<K\varepsilon+\alpha.$
\end{lemma}
\begin{proof}
Take any full component of intersection \( \phi_t^I\in\Phi_t. \) Assume
that it intersects 
the front end of $A$.
We cut $A$
along flow lines in 
\(
n:=\left\lfloor {1}/{\varepsilon }\right\rfloor \)
 pieces
\[A_{i}:=\left\{v\in A:\tau (v)\in \left[\frac{i\varepsilon}{n},
\frac{(i+1)\varepsilon
}{n}\right)\right\}\]
 of equal measure 
 (\( i=0,\dots ,n-1 \)). By the mixing property, \( m(A_{i}\cap g^{t}A_{0}) \)
is asymptotically
independent of \( i \) as \( t\rightarrow \infty . \) Hence the number of full
components of intersection of $A_i$ with
 $g^tA_0$ is asymptotically independent
of \( i. \) Since any
 intersection component of \( A_{i}\cap g^{t}A_{0} \) 
 has depth
\( \tau  \) with 
\(
|\tau -i\varepsilon /n|<\varepsilon /n,\)
 we see that the average 
 of $\theta$
 is \( \varepsilon /2 \) up to an error of order \( \varepsilon ^{2}. \)

The same reasoning applies if \( A_{0} \) is changed to \( A_{n-1} \), hence
 for $  \phi_t^I $ intersecting 
 the back end of $A$ instead of the front end.
\end{proof}
Note that if we compute the measure of an intersection $A_0\cap
g^tA_{n-1}$ for $t$ large, 
the terms which are not in full components of intersection
contribute only a fraction which by mixing is asymptotically zero because
$m(A_t')\cong m(A)$,
i.e., 
\(\exists K<\infty\ \forall\alpha>0\ \exists \varepsilon_{0}>0
\ \forall \varepsilon\in(0,\varepsilon_{0}),\ A=A(\varepsilon)
\ \exists t_{0}>0\ \forall t>t_0:
\left|\ln (m(A_t')  / m(A) )\right| 
<K\varepsilon+\alpha\).
This follows from
$$m(\{v\in A:s(v)<S(t)\})\to 0\text{ as }t\to\infty$$ and
$$m(\{v\in A:\tau(v)\in[0,\varepsilon^2]
\cup[\varepsilon-\varepsilon^2,\varepsilon]\})= 2\varepsilon m(A).$$
\subsection{Counting intersections}
\begin{thm}[Few intersection components through the stable
end]\label{thm:stable-boundary-errors-are-small}
The number $\#\Phi^s_t$ of intersection components that touch the stable end 
$\partial^s A$
 is asymptotically a zero proportion of the number of all boundary components: 
 $$\frac{\#\Phi^s_t}{N(A,t)}\cong 0.$$
In other words,
 \(\exists K<\infty\ \forall\alpha>0\ \exists \varepsilon_{0}>0
\ \forall \varepsilon\in(0,\varepsilon_{0}), A=A(\varepsilon)
\ \exists t_{0}>0\ \forall t>t_0:
 \frac{\#\Phi^s_t}{N(A,t)}<K\varepsilon+\alpha.\)
\end{thm}
\begin{proof}
Let $F:=g^{[0,\varepsilon]}\overline{B^u_{\delta}}(v_0).$
First note that 
\[
m(\phi_t^I)\bowtie \frac{\theta (\phi_t^I)}{\varepsilon }m^{0u}(F)\sigma (t)
\]
 for  \( \phi_t^I\in\Phi_t \setminus\Phi_t^s\),
i.e., $\exists K,\varepsilon_0>0
\ \forall \varepsilon\in(0,\varepsilon_0),\ A=A(\varepsilon)
\ \exists t_0\ \forall t>t_0: $
\[
\left|\ln
\frac{\varepsilon m(\phi_t^I)}
{\theta (\phi_t^I)m^{0u}(F)\sigma (t)}
\right|
<K\varepsilon.
\]
This is so since by Lemma \ref{lem:sizeofstablefibers}
  the stable measure of the pieces of
stable fibers in \( \phi_t^I \) is equal to \( \sigma (t) \) up to an error term 
that converges to 0 as \( \varepsilon \to 0 \)
and since 
by holonomy invariance (Theorem \ref{thm:holonomyinvariance})
 and by Lemma \ref{lem:side-to-side}
 the \( m^{0u} \)-measure of $0u$-leaves of \( \phi_t^I \) is the same as that
of \( F \), except that
the thickness of the intersection is not 
\( \varepsilon  \) but \( \theta (\phi_t^I). \)

For $s$-holonomic $p,q$, i.e. $p\in W^s(q)$,
the bounded subsets $B^{0u}_r(p)$ and $B^{0u}_r(q)$
get arbitrarily close under the flow ``pointwise except at the boundary'' 
in the following sense: there exists 
$R_1:= d^s(p,q)$ with $d\left(g^t p,B_{R_1}^{0u}(g^t q)\right)\to 0$ as $t\to\infty$. 
Moreover,
if we write $H$ for the holonomy
map from $B^{0u}_r(p)$ to $B^{0u}_r(q)$ along stable fibers, then
 for $R_2>R_1$ the convergence 
$d\left(g^t p',B_{R_1}^{0u}(g^t H(p'))\right)\to 0$ as $t\to\infty$ is 
uniform in $p'$ for all $p'\in B_{R_2-R_1}^{0u}(p)$.
See \cite{Gun}  for a proof of these claims.
Thus there exists $D^s=D^s(t):[0,\infty)\to (0,\infty)$ with
$D^s(t)\to 0$ as $t\to\infty$ such that 
$\phi\in\Phi_t^s$ implies $\phi\subset B^s_{D^s(t)}\left(\partial^sA\right)$.

Existence of a decomposition of $m$ into conditionals (Proposition \ref{prop:condKn})
and their holonomy invariance (Theorem \ref{thm:holonomyinvariance})
imply that
$m\left(B^s_{D}(\partial^sA)\right)\to 0$ as $D\to 0$.

For $\phi_t^I\in\Phi_t^s$ define 
$$\widehat{\phi}_t^I{}:=
g^{[-\varepsilon,\varepsilon]}(I)\cap B^s_{D^s(t)}(A)\cap g^tA.$$
This differs from $\phi_t^I$ by extending it in the stable direction 
beyond the stable boundary of $A$. 
We could also have written $\widehat{\phi}_t^I{}=
g^{[-\varepsilon,\varepsilon]}(I)\cap B^s_{D^s(t)}\left(\partial^s A\right)\cap g^tA.$
The set
$\widehat{\phi}_t^I$ is the intersection of $g^tA$ not only with $A$ itself,
but with a stable neighborhood of $A$; this allows us to treat 
$\widehat{\phi}_t^I\in\Phi_t^s$ like the elements ${\phi}_t^I\in\Phi_t$.
Namely, for such $\phi_t^I\in\Phi_t^s$, the formula
\(
m\left(\widehat\phi_t^I\right)\bowtie 
{\theta \left(\widehat\phi_t^I\right)}m^{0u}(F)\sigma (t)/{\varepsilon }
\)
still holds, by the same argument as in the case of $\phi_t^I\in\Phi_t$.
Since 
$\theta \left(\widehat{\phi_t^I}\right)
\leq
\theta \left(\phi_t^I\right)+\varepsilon^2$
and
$\theta \left(\phi_t^I\right)/\varepsilon
\leq 1$, this shows that 
\(
m\left(\widehat\phi_t^I\right)\leq \text{const}\cdot e^{-ht}.
\)
Therefore 
$$
\#\Phi_t^s/\#\Phi_t \leq \text{const}\cdot
m\left(B^s_{D^s(t)}\left(\partial^sA\right)\right)\to 0 
\text{ as }t\to\infty,
$$ proving the claim.
\end{proof}
\begin{remark}
The proof of Theorem \ref{thm:stable-boundary-errors-are-small} would be much
shorter and very easy 
if distances in the stable direction would contract uniformly, as they do for
uniformly hyperbolic systems. But in our case they do not. In fact, they do not
necessarily even convege to zero.
\end{remark}
\begin{prop}
\label{prop:N(A,t)}
The number 
$N(A,t)$ 
satisfies
\[
N(A,t)\cong 2e^{ht}m(A).\]
In other words,
 \(\exists K<\infty\ \forall\alpha>0\ \exists \varepsilon_{0}>0
\ \forall \varepsilon\in(0,\varepsilon_{0}),\ A=A(\varepsilon)
\ \exists t_{0}>0\ \forall t>t_0:
\left|\ln (N(A,t)/2e^{ht}m(A))\right| 
<K\varepsilon+\alpha.\)
\end{prop}
\begin{proof}
As in the proof of Theorem \ref{thm:stable-boundary-errors-are-small}, writing
 $F:=g^{[0,\varepsilon]}\overline{B^u_{\delta}}(v_0)$
 gives the estimate 
\[
m\left(\phi_t^I\right)\bowtie \frac{\theta \left(\phi_t^I\right)}{\varepsilon }m^{0u}(F)\sigma (t).
\]

Since by Lemma \ref{lem:avgthickness}
the average of the \( \theta (\phi_t^I)\, 
\textrm{is asymptotically }\varepsilon /2, \)
we get
\[
\frac{1}{N(A,t)}\sum_{\phi_t^I\in\Phi_t}
m\left(\phi_t^I\right)\cong \frac{1}{2}\sigma (t)m^{0u}(F),
\]
i.e., $\exists K<\infty\ \forall \alpha>0\ \exists\varepsilon_0>0
\ \forall \varepsilon\in(0,\varepsilon_0),\ A=A(\varepsilon)
\ \exists t_0\ \forall t>t_0: $
\[
\left|\ln
\frac{2\sum_{\phi_t^I\in\Phi_t}m\left(\phi_t^I\right)}
{N(A,t)\sigma (t)m^{0u}(F)}
\right|
<K\varepsilon+\alpha.
\]

Since the measure of \( A\cap g^{t}A \) is 
asymptotically the sum of the measures of the
full components
of intersection (there are \( N(A,t) \) of those), we obtain
\begin{equation}\label{eqnNsm}
m(A\cap g^{t}A)\cong \frac{1}{2}N(A,t)\sigma (t)m^{0u}(F),
\end{equation}
or, in more detail, $\exists K<\infty\ \forall \alpha>0\ \exists\varepsilon_0>0
\ \forall \varepsilon\in(0,\varepsilon_0),\ A=A(\varepsilon)
\ \exists t_0\ \forall t>t_0: 
\left|\ln
({2m(A\cap g^{t}A)}
/{N(A,t)\sigma (t)m^{0u}(F)})
\right|
<K\varepsilon+\alpha.$

Moreover note that by the mixing property of \( g \),
\[
m(A\cap g^{t}A)\cong e^{ht}m(A)\sigma (t)m^{0u}(F),
\]
i.e., $\exists K<\infty\ \forall \alpha>0\ \exists\varepsilon_0>0
\ \forall \varepsilon\in(0,\varepsilon_0),\ A=A(\varepsilon)
\ \exists t_0\ \forall t>t_0: 
\left|\ln
({m(A\cap g^{t}A)}
/{e^{ht}\sigma (t)m(A)m^{0u}(F)})
\right|
<K\varepsilon+\alpha.$

 The claim follows from combining these two estimates. 
\end{proof}
\subsection{Bounds on error terms in intersection counting}
\begin{defn}\label{def:G(t,e)}
Define ${\bf G}(t,\varepsilon):=\{$all geometric orbit
segments in $A$ of periodic orbits 
with period in $[t-\varepsilon,t+\varepsilon]\}$.
Let $G(t,\varepsilon):=\#{\bf
G}(t,\varepsilon)$. 
For
better comparison we write $N(t,\varepsilon):=N(A,t)=\#\Phi_t(\varepsilon)$ 
where $\Phi_t=\Phi_t(\varepsilon)$ is as before the
set of all full intersection components for given $t,\varepsilon.$ 
\end{defn}
\begin{prop}\label{prop:segs-asym-comps}
The number of orbit segments passing through $A$ that belong to periodic orbits
with period
in $[t-\varepsilon,t+\varepsilon]$ is $\cong N(A,t)$.
I.e., $$N(t,\varepsilon)\cong G(t,\varepsilon),$$ or more formally:
$\exists K<\infty\ \forall \alpha>0\ \exists\varepsilon_0>0
\ \forall \varepsilon\in(0,\varepsilon_0),\ A=A(\varepsilon)
\ \exists t_0\ \forall t>t_0: $
\[
\left|\ln\frac
{G(t,\varepsilon)}
{N(t,\varepsilon)}
\right|
<K\varepsilon+\alpha.
\]
\end{prop}
\begin{proof} 
By Lemma \ref{lem:segtocomp} we have
a map ${\bf G}(t,\varepsilon-2\varepsilon^2)\to\Phi_t(\varepsilon)$ 
and by Lemma \ref{lem:comptoseg} a map
$\Phi_t(\varepsilon)\setminus\Phi_t^s(\varepsilon)\to{\bf G}(t,\varepsilon)$. 
These maps are invertible between their domains
and images; hence they are injective. Thus we have 
$$G(t,\varepsilon-2\varepsilon^2)\leq
N(t,\varepsilon)\leq G(t,\varepsilon).$$

Since $N(t,\varepsilon-2\varepsilon^2)\leq 
G(t,\varepsilon-2\varepsilon^2),$ it suffices to show 
$$
N(t,\varepsilon)\cong N(t,\varepsilon+\varepsilon^2),
$$
i.e., $\exists K<\infty\ \forall \alpha>0\ \exists\varepsilon_0>0
\ \forall \varepsilon\in(0,\varepsilon_0),\ A=A(\varepsilon)
\ \exists t_0\ \forall t>t_0: $
\[
\left|\ln\frac
{N(t,\varepsilon)}
{N(t,\varepsilon+\varepsilon^2)}
\right|
<K\varepsilon+\alpha.
\]

Partition $A$ again into
\( n:=\left\lfloor {1}/{\varepsilon }\right\rfloor \)
 pieces 
\[A_{i}:=\left\{v\in A:\tau (v)\in \left[{i\varepsilon}/{n},{(i+1)\varepsilon
}/{n}\right)\right\}\] of equal measure
(\( i=0,\dots ,n-1 \)).
Mixing
implies that 
$$
m(A_0\cap g^tA_{n-1})\cong \varepsilon^2m(A)^2,
$$
i.e., $\exists K<\infty\ \forall \alpha>0\ \exists\varepsilon_0>0
\ \forall \varepsilon\in(0,\varepsilon_0),\ A=A(\varepsilon)
\ \exists t_0\ \forall t>t_0: $
\(
\left|\ln(
{m(A_0\cap g^tA_{n-1})}
/{\varepsilon^2m(A)^2}
)\right|
<K\varepsilon+\alpha.
\)

Observe that 
in analogy to equation (\ref{eqnNsm}) we have
$$
\varepsilon^2m(A)^2\cong m(A_0\cap
g^{t+\varepsilon^2}A_{n-1})\cong
\frac{1}{2}\varepsilon^2N(t,\varepsilon)m^{0u}(F)\sigma(t),
$$
i.e., $\exists K<\infty\ \forall \alpha>0\ \exists\varepsilon_0>0
\ \forall \varepsilon\in(0,\varepsilon_0),\ A=A(\varepsilon)
\ \exists t_0\ \forall t>t_0: 
\left|\ln
(
2m(A)^2
/
N(t,\varepsilon)m^{0u}(F)\sigma(t)
)
\right|
<K\varepsilon+\alpha.$

The full components of $A_0\cap g^{t+\varepsilon^2}A_{n-1}$ 
which are newly
created at the back end of $A$ by increasing $t$ to $t+\varepsilon^2$ have
average thickness $\varepsilon^2/2$ and hence
average measure $\frac{1}{2}\varepsilon m^{0u}(F)\sigma(t).$ 
Hence this increase of
$t$ to $t+\varepsilon^2$ can produce at most $\cong\varepsilon N(t,\varepsilon)$ such
full components. Thus 
$$N(t+\varepsilon^2,\varepsilon)\stackrel{\sim}{\leq} 
N(t,\varepsilon)+\varepsilon
N(t,\varepsilon),$$
where the notation
$f_1(t,\varepsilon)\stackrel{\sim}{\leq}f_2(t,\varepsilon)$ means 
$f_1(t,\varepsilon){\leq}f_3(t,\varepsilon){\cong}f_2(t,\varepsilon)$ 
for some $f_3$, i.e.,
$\exists K<\infty\ \forall \alpha>0\ \exists\varepsilon_0>0
\ \forall \varepsilon\in(0,\varepsilon_0),\ A=A(\varepsilon)
\ \exists t_0\ \forall t>t_0: $
$$\ln
\frac{
f_1(t,\varepsilon)
}{
f_2(t,\varepsilon)
}
<K\varepsilon+\alpha.$$
In other words, we have shown 
$\exists K<\infty\ \forall \alpha>0\ \exists\varepsilon_0>0
\ \forall \varepsilon\in(0,\varepsilon_0),\ A=A(\varepsilon)
\ \exists t_0\ \forall t>t_0: $
$$\ln
\frac{
N(t+\varepsilon^2,\varepsilon)
}{
N(t,\varepsilon)+\varepsilon
N(t,\varepsilon)
}
<K\varepsilon+\alpha.$$

 It follows that $$N(t+\varepsilon^2,\varepsilon){\cong} N(t,\varepsilon).$$

Since increasing $\varepsilon$ by $\varepsilon^2$ 
leads to a gain in the number of full
components by making more of them enter the back end of the flow 
cube exactly like
increasing $t$ by  $\varepsilon^2$ does, plus a similar increase in number by
making some of them delay their departure through the front end of the
flow cube, we get 
$$N(t,\varepsilon+\varepsilon^2)\stackrel{\sim}
{\leq} N(t,\varepsilon)+2\varepsilon 
N(t,\varepsilon),$$ 
i.e., we have shown 
$\exists K<\infty\ \forall \alpha>0\ \exists\varepsilon_0>0
\ \forall \varepsilon\in(0,\varepsilon_0),\ A=A(\varepsilon)
\ \exists t_0\ \forall t>t_0: $
$$\ln
\frac{
N(t,\varepsilon+\varepsilon^2)
}{
N(t,\varepsilon)+2\varepsilon 
N(t,\varepsilon)
}
<K\varepsilon+\alpha.$$

This shows that 
$N(t,\varepsilon+\varepsilon^2){\cong} N(t,\varepsilon).$ 
Hence $N(t,\varepsilon)\cong G(t,\varepsilon)$ as claimed.
\end{proof}
\subsection{A Bowen-type property of the measure of maximal entropy}
\begin{defn}
Let \( P_{t} \) be the number of homotopy classes of closed geod\-esics of length
at most \( t. \) Let \( P_{t}(A) \) be the number of closed
geodes\-ics of length at most \( t \) that intersect \( A. \) Let \( P_{t}' \)
be the number of \emph{regular} closed geodesics of length at most \( t. \) 
\end{defn}
\begin{remark}[Terminology] When we say ``closed geodesic'', we mean 
``periodic orbit  for the geodesic
flow'', i.e. with parameterization (although
always by arclength and modulo adding a constant to the parameter). 
Thus a locally shortest curve is counted as two geodesics
(i.e. periodic orbits for the geodesic flow), namely one for each direction. 
\end{remark}
\( P_{t}',P_{t}(A) \) are finite because there is only one regular geodesic in
each
homotopy class.
Clearly 
\[
P_{t}(A)\leq P_{t}'\leq P_{t}\]
 for any \( t. \) We will show that these are in fact
asymptotically equal.
\begin{lemma}
\[
P_{t}'\sim P_{t}.\]
\end{lemma}
\begin{proof}
Singular geodesics have a smaller exponential growth rate than regular ones
because the  entropy of the singular
set is smaller than the topological entropy  (\cite{knie gafa})
whereas the entropy of the regular
set equals the topological
entropy.
\end{proof}
In the case that \( M \) is a surface, the growth rate of \(
\textrm{{\bf{Sing}}} \)
is in fact zero, since the existence of
a parallel perpendicular Jacobi field implies that the
largest
Liapunov exponent is zero.
\begin{defn}
Let \( \mu _{t} \) be the arclength measure on all regular periodic orbits
of length at most \( t, \) normalized to 1:
\begin{eqnarray*}
&\mathbf{P}_{t}&:=\{\textrm{regular closed geodesics of length}
\leq t\},\\
&\ \mathbf{P}_{t}(A)&:=\{\textrm{geodesics in }\mathbf{P}_{t}
\textrm{ which pass}\textrm{ through }A\},\\
&\mu _{t}&:=\frac{1}{\#\mathbf{P}_{t}}\sum _{c\in \mathbf{P}_{t}}
\frac{1}{\textrm{len}(c)}\delta _{c},\\
&\mu ^{A}_{t}&:=\frac{1}{\#\mathbf{P}_{t}(A)}\sum _{c\in
\mathbf{P}_{t}(A)}
\frac{1}{\textrm{len}(c)}\delta _{c}.
\end{eqnarray*}
Here $\delta _{c}$ is the length measure on $\dot c$. 
\end{defn}
\begin{thm}
\label{thm:m asym mu}Any weak limit \( \mu  \) of \( (\mu _{t})_{t>0} \)
is the measure of maximal entropy.
Moreover,  any weak limit \( \mu ^{A} \) of \( (\mu ^{A}_{t})_{t>0} \) 
is the measure of maximal entropy.
\end{thm}
 In other words, for any $t_k\to\infty$ such that
 \( (\mu ^{A}_{t_{k}})_{t_{k}\in \mathbb
{R}} \) converges weakly
and for any measurable \( U \) the following holds: 
\[
\lim _{k\rightarrow \infty }\mu ^{A}_{t_{k}}(U)=m(U).\]
 Similarly with \( \mu ^{A} \) replaced by \( \mu . \)

\begin{proof}
Knieper showed in \cite{knie ann} that \( m \) 
can be obtained as a weak limit of
the measures \( \mu _{t_{k}} \) 
which are Borel probability measures supported
on \( \mathbf{P}_{t_{k}} \); see also 
\cite{Pol}.  The singular closed geodesics can be neglected
because the singular set has entropy smaller than \( h. \) Hence any weak limit
of \( \mu _{t} \) equals \( m. \)

Since 
\[
P_{t}(A)\geq C\frac{e^{ht}}{t}\]
 (\cite[Remark after Theorem 5.8]{knie ann}), any weak limit of the measures
\( \mu ^{A}_{t_{k}} \) concentrated on \( \mathbf{P}_{t_{k}}(A) \) has entropy
\( h. \) Since the measure of maximal entropy is unique, any such weak limit
equals \( m \).
\end{proof}
\begin{remark}
\label{rem:m asym mu}This means that we can approximate the measure of 
maximal entropy \( m \)
of a measurable set  by its \( \mu _{t_{k}} \)-measure for \( k \)
sufficiently large. Moreover, when counting orbits, an arbitrarily  small
regular local product
cube \( A \) will suffice to count periodic orbits in such a way that the
fraction
of those not counted will converge to zero as the period of these orbits
becomes
large. We use this fact in the proof of Theorem \ref{thm:P t,e}.
\end{remark}
\begin{corollary}
\(
P_{t}(A)\sim P_{t}.\)
\end{corollary}
\begin{proof}
By theorem \ref{thm:m asym mu} 
the measure on the geodesics in $\mathbf{P}_t\setminus \mathbf{P}_t(A)$ (which
assigns zero measure to $A$) would otherwise also converge weakly to the
measure of maximal entropy.
\end{proof}

\begin{defn}
Let \( P_{t,\varepsilon } \) be the number of {regular} geodesics
with length in \( (t-\varepsilon ,t+\varepsilon ] \). 

\end{defn}
Again, \( P_{t,\varepsilon } \) is finite because there is only one regular geodesic in
each
homotopy class.
\begin{thm}
\label{thm:P t,e}The number \( P_{t,\varepsilon } \)  
of regular closed geodesics with prescribed length
is given by the asymptotic formula
\[
P_{t,\varepsilon }\cong \frac{\varepsilon N(A,t)}{t\cdot m(A)}.
\]
I.e. $\exists K<\infty\ \forall \alpha>0\ \exists\varepsilon_0>0
\ \forall \varepsilon\in(0,\varepsilon_0),\ A=A(\varepsilon)
\ \exists t_0\ \forall t>t_0: $
\[
\left|\ln\frac
{P_{t,\varepsilon }\cdot t\cdot m(A)}
{\varepsilon N(A,t)}
\right|
<K\varepsilon+\alpha.
\]
\end{thm}
\begin{proof}
By Theorem \ref{thm:m asym mu}, for a typical closed geodesic \( c \) with 
sufficiently large length,
\[
\frac{1}{\textrm{len}(c)}\delta_c(A)=
\frac{1}{\textrm{len}(c)}\int _{\dot{c}\cap A}d\textrm{len}\cong m(A).\]
 Here ``typical'' means that the number of closed geodesics of length at most
\( t \) that have this property is asymptotically the same as the number of
all closed geodesics of length at most \( t; \) in other words, the ratio
tends to 1. That means: $\exists K<\infty\ \forall \alpha>0\ \exists\varepsilon_0>0
\ \forall \varepsilon\in(0,\varepsilon_0),\ A=A(\varepsilon)
\ \exists t_0\ \forall t>t_0: $
$$
\left|\ln\left({
\frac{1}{m(A)\cdot\#\mathbf{P}_{t}}\sum _{c\in \mathbf{P}_{t}}
\frac{1}{\textrm{len}(c)}\delta _{c}
}{
}
\right)\right|<K\varepsilon+\alpha.
$$

Hence such a geodesic of length $t$ (which consists of \( t/\varepsilon  \) 
segments of length \( \varepsilon 
\))
will have asymptotically \( m(A)t/\varepsilon  \) segments 
of length \( \varepsilon  \)
intersecting
\( A. \) Thus 
\[
P_{t,\varepsilon }\cong \frac{\varepsilon G(t,\varepsilon)}{tm(A)},
\]
i.e., $\exists K<\infty\ \forall \alpha>0\ \exists\varepsilon_0>0
\ \forall \varepsilon\in(0,\varepsilon_0),\ A=A(\varepsilon)
\ \exists t_0\ \forall t>t_0:$ $$ 
\left|\ln\frac{
tm(A)P_{t,\varepsilon }
}{
\varepsilon G(t,\varepsilon)
}\right|
<K\varepsilon+\alpha,$$
where $G$ is as in  Definition \ref{def:G(t,e)}. 
The statement of 
Proposition \ref{prop:segs-asym-comps} then shows the claim.
\end{proof}
\begin{remark}
It suffices to consider closed orbits which are not multiple
iterates of some other closed orbit for the following reason:
If $H(t,k)$ is the number of periodic orbits passing through $A$ of length at
most $t$ which are $k$-fold iterates, then $H(t,k)=0$ for $k>t/\text{inj}(M)$.
Thus the number of  segments of $A$ which are transversed by all multiple
iterates is at most $\sum_{k=2}^{\lfloor t/\text{inj}(M)\rfloor} kH(t,k).$ By
Knieper's multiplicative estimate (see Section \ref{kniepersection}),
 this number is at most
$\text{const}\cdot\sum_{k=2}^{\lfloor t/\text{inj}(M)\rfloor} ke^{ht/k}$, thus
at most 
$\text{const}\cdot t^2e^{ht/2}.$ This contributes only a zero
asymptotic 
fraction of the segments and can thus be ignored.
\end{remark}
Proposition \ref{prop:N(A,t)} and Theorem \ref{thm:P t,e} combined yield:
\begin{corollary}\label{cor:usenext}
\[
P_{t,\varepsilon }(A)\cong \frac{2\varepsilon e^{ht}}{t}.
\]
I.e., $\exists K<\infty\ \forall \alpha>0\ \exists\varepsilon_0>0
\ \forall \varepsilon\in(0,\varepsilon_0),\ A=A(\varepsilon)
\ \exists t_0\ \forall t>t_0: $
\(
\left|\ln(
{tP_{t,\varepsilon }(A)}
/{2\varepsilon e^{ht}}
)\right|
<K\varepsilon+\alpha.
\)
\hspace*{\fill}\qed
\end{corollary}
\subsection{Proof of the main result}
The desired asymptotic formula is now derived:
\begin{thm}[Precise asymptotics for periodic orbits]
\label{thm:P asym e^ht/ht}Let \( M \) be a compact Riemannian manifold of
nonpositive
curvature whose rank is one. Then the number \( P_{t} \) of homotopy classes
of periodic orbits of length at most \( t \) for the geodesic flow  is
asymptotically given by the
formula
\[
P_{t}\sim \frac{e^{ht}}{ht}\]
 where \( \sim  \) means that the quotient converges to \( 1 \) as \( t\to
\infty . \)
\end{thm}
\begin{proof}
We use the standard limiting process 
\[
\int _{a}^{b}f(x)dx\bowtie \sum _{i=\lfloor a/2\varepsilon \rfloor }
^{\lfloor b/2\varepsilon
\rfloor }2\varepsilon f((2i+1)\varepsilon )
\]
 for suitable functions \( f \) (in particular, if \( f \) is continuous and
piecewise
monotone, as is the case for $f(x)=e^{hx}/x$); 
since this is elementary, ``$\bowtie$'' requires no explanation in long form
here. 
Choose some fixed sufficiently
large
 number \( t_0>0. \) Since
we can ignore all closed geodesics of length at most \( t_0 \) for the
asymptotics,
we see that for \( t>t_0 \) by Corollary \ref{cor:usenext} we get
\begin{eqnarray*}
P_{t}' & \cong  & P_{t}(A)
  \cong   \sum _{i=\lfloor t_0/2\varepsilon \rfloor }
 ^{\lfloor t/2\varepsilon \rfloor
}P_{(2i+1)\varepsilon ,\varepsilon }
  \cong   \sum _{i=\lfloor t_0/2\varepsilon \rfloor }
 ^{\lfloor t/2\varepsilon \rfloor
 }2\varepsilon \frac{e^{h(2i+1)\varepsilon}}{(2i+1)\varepsilon}\\
 & \cong  & \int _{t_0}^{t}\frac{e^{hx}}{x}dx
\   =\   \left.\frac{e^{hx}}{hx}\right|_{t_0}^{t}+\int
_{t_0}^{t}\frac{e^{hx}}{hx^{2}}dx
\  \cong \   \frac{e^{ht}}{ht}-\frac{e^{ht_0}}{ht_0}\\
 & \cong  & \frac{e^{ht}}{ht}.
\end{eqnarray*}
In more detail, $\exists K<\infty\ \forall \alpha>0\ \exists\varepsilon_0>0
\ \forall \varepsilon\in(0,\varepsilon_0),
\ \exists t_0\ \forall t>t_0: $ 
$$
\left|\ln
\frac{ htP_t'}
{ e^{ht}}
\right|
<K\varepsilon+\alpha.
$$

Note that
there is no longer any dependence
on \( \varepsilon  \)   and that  \( P_{t}'\sim P_{t} \). 
Hence\[
P_{t}\sim \frac{e^{ht}}{ht}.\]This concludes the proof.
\end{proof}


\begin{thebibliography}{PoSh2}
\bibitem[Ana1]{Ana1}\noun{N.~Anantharaman}: \emph{Precise counting for closed
orbits of Anosov flows.} Ann. Sci. {\'E}cole Norm. Sup. (4) 33 (2000), no. 1,
33--56.
\bibitem[Ana2]{Ana2}\noun{N.~Anantharaman}: 
\emph{Counting geodesics which are optimal in homology.} 
Ergodic Theory Dynam. Systems 23 (2003), no. 2, 353--388.
\bibitem[Ano]{Ano}\noun{D.~V.~Anosov:} \emph{Geodesic flows on closed
Riemannian manifolds
of negative curvature.} Trudy Mat. Inst. Steklov \textbf{90} (1967) (Russian);
Proceedings of the Steklov Institute of Mathematics, No. 90 (1967) (English).
\bibitem[Bab]{Bab}\noun{M.~Babillot}: \emph{On the mixing property for
hyperbolic systems}. Israel J. Math. 129 (2002), 61--76. 
\bibitem[BaLe]{BaLe}\noun{M.~Babillot; F.~Ledrappier}:
\emph{Lalley's theorem on periodic orbits 
of hyperbolic flows.}
 Ergodic Theory Dynam. Systems \textbf{18} (1998), no. 1, 17--39. 
\bibitem[Bal1]{higher rk}\noun{W.~Ballmann}: \emph{Nonpositively curved
manifolds of higher rank.}
Ann. of Math. \textbf{122} (1985), 597--609.
\bibitem[Bal2]{dmv}\noun{W.~Ballmann}: \emph{Lectures on spaces of
nonpositive curvature.}
With an appendix by Misha Brin. DMV Seminar, 25. Birkh\"auser Verlag, Basel,
1995.
\bibitem[Bal3]{Bal3}\noun{W.~Ballmann}: \emph{Der Satz von Lusternik und
Schnirelmann.} 
Bei\-tr\"age zur Differentialgeometrie, Heft 1. 
Bonner Mathematische Schriften, 102, 1--25.
Universität Bonn, Mathematisches Institut, Bonn 1978. 
\bibitem[BBE]{bbe}\noun{W.~Ballmann}; \noun{M.~Brin; P.~Eberlein}: \emph{Structure
of manifolds of nonpositive curvature. I.} Ann. Math. \textbf{122} (1985),
171--203. 
\bibitem[BaEb]{BaEb}\noun{W.~Ballmann; P.~Eberlein:} \emph{Fundamental
groups of manifolds
of nonpositive curvature.} J. Differential Geom. \textbf{25} (1987), no. 1,
1--22. 
\bibitem[BGS]{5}\noun{W.~Ballmann; M.~Gromov; V.~Schroeder:}
\emph{Manifolds
of nonpositive curvature}. Progress in Mathematics, 61. Birkhäuser Boston,
Inc.,
Boston, MA, 1985.
\bibitem[BTZ1]{BTZ1}\noun{W.~Ballmann;
G.~Thorbergsson; W.~Ziller:}
\emph{Closed geodesics on positively curved manifolds}. Ann. Math (2)
\textbf{116}, 213--247 (1982).
\bibitem[BTZ2]{BTZ2}\noun{W.~Ballmann;
G.~Thorbergsson; W.~Ziller:}
\emph{Existence of closed geodesics on positively curved manifolds.} J.
Differential Geom.
\textbf{83} (1983), no. 2, 221--252.
\bibitem[Ban]{Ban}\noun{V.~Bangert}: \emph{On the existence of closed
geodesics on  two-spheres.} Internat. J. Math \textbf{4} (1993) no. 1, 1--10.
\bibitem[BaPe]{BaPe}\noun{L.~Barreira; Ya.~Pesin:} \emph{Lectures on
Lyapunov exponents
and smooth ergodic theory}. University Lecture Series, 23. American
Mathematical
Society, Providence, RI, 2002.
\bibitem[BrPe1]{BrPe1}\noun{%
M.~I.~Brin; Ya.~B.~Pesin}: \emph{Partially hyperbolic dynamical systems.} (Russian) 
Uspehi Mat. Nauk  28  (1973), no. 3(171), 169--170.
\bibitem[BrPe2]{BrPe2}\noun{%
M.~I.~Brin; Ya.~B.~Pesin}: \emph{Partially hyperbolic dynamical systems.} (Russian)
Izv. Akad. Nauk SSSR Ser. Mat.  38  (1974), 170--212. 
\bibitem[BuKa]{BuKa}\noun{K.~Burns; A.~B.~Katok}.: \emph{Manifolds
with non-positive
curvature.} Ergodic Theory and Dynamical Systems \textbf{5} (1985), 307--317.
\bibitem[BuSp]{BuSp}\noun{K.~Burns; R.~Spatzier}: \emph{Manifolds of
nonpositive curvature
and their buildings}. Inst. Hautes Études Sci. Publ. Math. No. \textbf{65},
(1987), 35--59. 
\bibitem[Ebe1]{Ebe isom}\noun{P.~B.~Eberlein}: \emph{Isometry groups of
simply connected 
manifolds
of nonpositive curvature, II}. Acta Math. \textbf{149} (1982), 41--69
\bibitem[Ebe2]{Ebe2 visibility}\noun{P.~B.~Eberlein}: \emph{Geodesic
flows on negatively 
curved manifolds}.
\emph{I}. Annals of Mathematics (2) \textbf{95} (1972), 492--510.
\bibitem[Ebe3]{Ebe3 anosov}\noun{P.~B.~Eberlein}: \emph{When is a
geodesic flow} \emph{of 
Anosov
type. I, II.} J. Differential Geometry \textbf{8} (1973), 437-463.
\bibitem[Ebe4]{Ebe4 anosov}\noun{P.~B.~Eberlein}: \emph{When is a
geodesic flow} \emph{of 
Anosov
type. I, II.} J. Differential Geometry  \textbf{8} (1973), 565--577.
\bibitem[Ebe5]{Ebe5}\noun{P.~B.~Eberlein}: \emph{Structure of manifolds
of nonpositive curvature.}
Global differential geometry and global analysis 1984 (Berlin, 1984), 86--153,
Lecture Notes in Math., 1156, Springer, Berlin, 1985.
\bibitem[Esch]{Esch}\noun{J.-H.~Eschenburg}: \emph{Horospheres and the
stable part of
the geodesic flow}. Math. Z. \textbf{153} (1977), no. 3, 237--251.
\bibitem[Fra]{Fra}\noun{Franks, John}: \emph{Geodesics on $S\sp 2$ and periodic
points of
annulus homeomorphisms.}  Invent. Math.  108  (1992),  no. 2, 403--418. 
\bibitem[Gro]{gro}\noun{M.~Gromov}: \emph{Manifolds of negative
curvature}. J. Diff. Geom.
\textbf{14} (1978), 223--230.
\bibitem[Gun]{Gun}\noun{R.~Gunesch}: \emph{Precise asymptotics for
periodic orbits of the geodesic flow in nonpositive curvature.} 
Ph.D. Thesis, Pennsylvania State
University, 2002.
\bibitem[Ham]{Ham}\noun{U.~Hamenst{\"a}dt:} \emph{A new description of the
Bowen-Margulis 
measure.}
Ergodic Theory Dynam. Systems \textbf{9} (1989), no. 3, 455--464. 
\bibitem[Has]{Has}\noun{B.~Hasselblatt}: \emph{A new construction of the
Margulis measure
for Anosov flows.} Ergodic Theory Dynam. Systems \textbf{9} (1989), no. 3,
465--468.
\bibitem[Jos1]{JosSphere}\noun{J.~Jost}:  \emph{A nonparametric
proof of the theorem of Lusternik 
and Schnirelman.}  Arch. Math. (Basel)  \textbf{53}  (1989),  no. 5, 497--509. 
\bibitem[Jos2]{JosCorr}\noun{J.~Jost}: 
\emph{Correction to: ``A nonparametric
proof of the theorem of Lusternik 
and Schnirelman''.}  Arch. Math. (Basel)  \textbf{56}  (1991),  no. 6,
624. 
\bibitem[Kai1]{Kai1}\noun{V.~A.~Kaimanovich} \emph{%
Bowen-Margulis and Patterson measures on negatively curved compact manifolds.} 
Dynamical systems and related topics (Nagoya, 1990), 223--232, 
Adv. Ser. Dynam. Systems, 9, 
World Sci. Publishing, River Edge, NJ, 1991. 
\bibitem[Kai2]{Kai2}\noun{V.~A.~Kaimanovich} \emph{%
Invariant measures of the geodesic flow and measures at infinity on negatively 
curved manifolds.}
Hyperbolic behaviour of dynamical systems (Paris, 1990). 
Ann. Inst. H. Poincaré Phys. Théor. 53 (1990), no. 4, 361--393.
\bibitem[Kat1]{Kat1}\noun{A.~B.~Katok:} \emph{Ergodic perturbations of
degenerate integrable Hamiltonian systems.}
Izw. Akad. Nauk SSSR Ser. Mat. \textbf{37} (1973), 539--576 (Russian); 
Izw. Math. USSR vol. 7 (1973) no. 3 (English).
\bibitem[Kat2]{Kat2}\noun{A.~B.~Katok:} \emph{Infinitesimal Lyapunov
functions, invariant
cone families and stochastic properties of smooth dynamical systems.} With the
collaboration of Keith Burns. Ergodic Theory Dynam. Systems \textbf{14} (1994)
no. 4, 757--785.
\bibitem[KaHa]{kh}\noun{A.~B.~Katok; B.~Hasselblatt}:
\emph{Introduction to the modern
theory of dynamical systems.} Encyclopedia of Mathematics and its Applications,
54. Cambridge University Press, 1995. 
\bibitem[KaHu]{KaHu}\noun{V.~Yu.~Kaloshin; B.~R.~Hunt:} \emph{A
stretched exponential
bound on the rate of growth of the number of periodic points for prevalent
diffeomorphisms
I}. Electronic Research Announcements of the AMS \textbf{7} (2001) 17--27.
\bibitem[Kli]{Kli}\noun{W.~Klingenberg}: \emph{Lectures on closed
geodesics.} 3rd ed. Mathematisches Institut, Universit\"at Bonn, Bonn 1977.
\bibitem[Kni1]{knie ann}\noun{G.~Knieper}: \emph{The uniqueness of the
measure of maximal 
entropy
for geodesic flows on rank \( 1 \) manifolds.} Ann. of Math. (2) \textbf{148}
(1998), no. 1, 291--314.
\bibitem[Kni2]{knie gafa}\noun{G.~Knieper}: \emph{On the asymptotic
geometry of nonpositively
curved manifolds}. Geom. Funct. Anal. \textbf{7} (1997), no. 4, 755--782.
\bibitem[Kni3]{knie arch d math}\noun{G.~Knieper}: \emph{Das Wachstum der
\"Aquivalenzklassen
geschlossener Geod\"atischer.} Archiv d. Math. \textbf{40} (1983), 559-568.
\bibitem[Kni4]{kni4}\noun{G.~Knieper}: \emph{Hyperbolic dynamics and
Riemannian geometry.} In: \noun{A.~B.~Katok; B.~Hasselblatt}:
Elsevier Handbook in Dynamical Systems, vol. 1A, to 
appear in  2002.
\bibitem[Lal]{Lal}\noun{S.~P.~Lalley}:
\emph{Closed geodesics in homology classes on surfaces of variable negative
curvature.}
Duke Math. J. 58 (1989), no. 3, 795--821.
\bibitem[Led]{Led}\noun{F.~Ledrappier}: \emph{Structure au bord des
Vari\'{e}t\'{e}s \`{a} courbure n\'{e}gative.} 
S\'{e}minaire de th\'{e}orie spectrale et g\'{e}ometrie 
Grenoble 1994-1995 (97--122).
\bibitem[LuFe]{LuFe}\noun{L.~Lusternik; A.~I.~Fet}: \emph{Variational
problems on closed manifolds.} Dokl. Akad. Nauk SSSR (N.S.) \textbf{81}
(1951). 
\bibitem[LuSch]{LuSch}\noun{L.~Lusternik; L.~Schnirelmann}: \emph{Sur le
probl\`eme de trois g\'eod\'esics fermees sur les surfaces de genre 0.} C. R.
Acad. Sci. Paris \textbf{189} (1929).
\bibitem[Man]{mann ent}\noun{A.~Manning}: \emph{Topological entropy for
geodesic flows.} 
Ann.
of Math. (2) \textbf{110} (1979), no. \textbf{}3, 567--573.
\bibitem[Mar1]{Mar1}\noun{G.~A.~Margulis:} Doctoral Dissertation,
Moscow state university,
1970.
\bibitem[Mar2]{Mar2}\noun{G.~A.~Margulis}: \emph{Certain
applications of ergodic theory
to the investigation of manifolds of negative curvature.} Functional Anal.
Appl.
\textbf{3} (1969), no. 4, 335--336.
\bibitem[Mar3]{Mar3}\noun{G.~A.~Margulis}: \emph{Certain measures
that are connected with}
\textbf{\emph{Y}}\emph{-flows on compact manifolds}. Functional Anal. Appl.
\textbf{4} (1970), 57--67.
\bibitem[Mar4]{Mar4}\noun{G.~A.~Margulis}: 
\emph{On Some Aspects of the Theory of Anosov Systems. With a survey by Richard
Sharp: Periodic orbits of hyperbolic flows.} Springer Verlag, 2004.
\bibitem[Mat]{Mat}\noun{H.-H.~Matthias}: \emph{Eine Finslermetrik auf
$S^2$ mit nur zwei geschlossenen Geod\"atischen.} 
Beitr\"age zur Differentialgeometrie, Heft 1. 
Bonner Mathematische Schriften, 102. 
Universität Bonn, Mathematisches Institut, Bonn 1978. 
\bibitem[PaPo]{PaPo}\noun{W.~Parry; M.~Pollicott}:
  \emph{Zeta functions and the periodic orbit
 structure of hyperbolic dynamics}. Astérisque no. 187-188 (1990).
\bibitem[Pes]{Pes}\noun{Ya.~B.~Pesin}: \emph{Lectures on partial hyperbolicity
and stable ergodicity.} Zurich Lectures in
 Advanced Mathematics. European Mathematical Society (EMS), Z{\"u}rich, 2004.
\bibitem[Pol]{Pol}\noun{M.~Pollicott}: \emph{Closed geodesic
distribution for manifolds of non-positive curvature.} Discrete Contin. Dynam.
Systems \textbf{2} (1996), no. 2, 153--161.
\bibitem[PoSh1]{PoSh1}\noun{M.~Pollicott; R.~Sharp}: \emph{Error 
 terms for closed orbits of hyperbolic flows.} Ergodic Theory Dynam. Systems 
 \textbf{21} (2001), no. 2, 545--562. 
\bibitem[PoSh2]{PoSh2}\noun{M.~Pollicott; R.~Sharp}: \emph{Asymptotic expansion for
closed orbits in homology classes.} Geom. Dedicata 87 (2001), no. 1--3,
123--160.
\bibitem[Pat]{Pat}\noun{S.~J.~Patterson}:
\emph{Lectures on measures on limit sets of Kleinian groups.}
Analytical and geometric aspects of hyperbolic space (Coventry/Durham, 1984),
281--323, 
London Math. Soc. Lecture Note Ser., 111, 
Cambridge Univ. Press, Cambridge, 1987. 
\bibitem[Rad1]{Rad1}\noun{H.-B.~Rademacher}:
\emph{Morse-Theorie und geschlossene Geod{\"a}ti\-sche.}
Habilitationsschrift, Rheinische Friedrich-Wilhelms-Universit\"at Bonn, Bonn,
1991.
\bibitem[Rad2]{Rad2}\noun{H.-B.~Rademacher}:
\emph{Morse-Theorie und geschlossene Geod\"ati\-sche} [Morse theory and
closed geodesics].
Bonner Mathematische Schriften [Bonn Mathematical Publications], 229. 
Universität Bonn, Mathematisches Institut, Bonn, 1992. 
\bibitem[Rad3]{Rad3}\noun{H.-B.~Rademacher}:
\emph{On a generic property of geodesic flows.} Math. Ann. \textbf{298},
(1994), 101--116.
\bibitem[Rob]{Rob}\noun{T.~Roblin}:
\emph{Ergodicity and equidistribution in negative curvature. 
(Ergodicité et équidistribution en courbure négative.)} (French)
[J] Mém. Soc. Math. Fr., Nouv. Sér. 95, 96 p. (2003). 
\bibitem[Sch]{Sch}\noun{V.~Schroeder:} \emph{Rigidity of nonpositively
curved graphmanifolds.}
Math. Ann. \textbf{274} (1986), no. 1, 19--26. 
\bibitem[Sul]{Sul}\noun{D.~Sullivan}: \emph{The 
density at infinity of a discrete group of hyperbolic motions.}
Inst. Hautes Études Sci. Publ. Math. No. 50 (1979), 171--202.
\bibitem[Tol]{Tol}\noun{C.~H.~Toll}: \emph{A multiplicative
asymptotic for the prime
geodesic theorem.} Ph.D. Dissertation, University of Maryland, 1984.
\bibitem[Yue1]{Y jdg}\noun{Ch.~B.~Yue}: \emph{Brownian motion on Anosov
foliations and 
manifolds
of negative curvature}. J. Differential Geom. \textbf{41} (1995), no. 1,
159--183.
\bibitem[Yue2]{Y etds}\noun{Ch.~B.~Yue}: \emph{On a conjecture of Green.}
Ergodic Theory 
Dynam.
Systems \textbf{17} (1997), no. 1, 247--252.
\bibitem[Zil]{Zil}\noun{W.~Ziller}: \emph{Geometry of the Katok
examples.}  Ergodic Theory
Dynam. Systems  \textbf{3}  (1983),  no. 1, 135--157. 
\end{thebibliography}
\end{document}